\documentclass[a4paper]{amsart}
\usepackage{verbatim}
\usepackage[dvips]{color}
\usepackage{amssymb}
\usepackage{latexsym}
\usepackage{amsmath}
 \usepackage{float}
\usepackage{mathrsfs} 

\allowdisplaybreaks

\newtheorem{theorem}{Theorem}[section]

\newtheorem{lemma}{Lemma}[section]

\theoremstyle{definition}
\newtheorem{definition}{Definition}

\floatstyle{ruled} \restylefloat{figure} \restylefloat{table}

\numberwithin{equation}{section}

\newcommand{\beq}{\begin{equation}}
\newcommand{\bea}[1]{\begin{array}{#1} }
\newcommand{\eeq}{ \end{equation}}
\newcommand{\ea}{ \end{array}}

\newcommand{\si}{\sigma}

\def \R  {{\mathbb {R}}}

\def \F  {{\mathscr {F}}}

\def\mean#1{\mathchoice%
          {\mathop{\kern 0.2em\vrule width 0.6em height 0.69678ex depth -0.58065ex
                  \kern -0.8em \intop}\nolimits_{\kern -0.4em#1}}%
          {\mathop{\kern 0.1em\vrule width 0.5em height 0.69678ex depth -0.60387ex
                  \kern -0.6em \intop}\nolimits_{#1}}%
          {\mathop{\kern 0.1em\vrule width 0.5em height 0.69678ex
              depth -0.60387ex
                  \kern -0.6em \intop}\nolimits_{#1}}%
          {\mathop{\kern 0.1em\vrule width 0.5em height 0.69678ex depth -0.60387ex
                  \kern -0.6em \intop}\nolimits_{#1}}}

\def\vintslides_#1{\mathchoice%
          {\mathop{\kern 0.1em\vrule width 0.5em height 0.697ex depth -0.581ex
                  \kern -0.6em \intop}\nolimits_{\kern -0.4em#1}}%
          {\mathop{\kern 0.1em\vrule width 0.3em height 0.697ex depth -0.604ex
                  \kern -0.4em \intop}\nolimits_{#1}}%
          {\mathop{\kern 0.1em\vrule width 0.3em height 0.697ex depth -0.604ex
                  \kern -0.4em \intop}\nolimits_{#1}}%
          {\mathop{\kern 0.1em\vrule width 0.3em height 0.697ex depth -0.604ex
                  \kern -0.4em \intop}\nolimits_{#1}}}

\newcommand{\aveint}[2]{\mathchoice%
          {\mathop{\kern 0.2em\vrule width 0.6em height 0.69678ex depth -0.58065ex
                  \kern -0.8em \intop}\nolimits_{\kern -0.45em#1}^{#2}}%
          {\mathop{\kern 0.1em\vrule width 0.5em height 0.69678ex depth -0.60387ex
                  \kern -0.6em \intop}\nolimits_{#1}^{#2}}%
          {\mathop{\kern 0.1em\vrule width 0.5em height 0.69678ex depth -0.60387ex
                  \kern -0.6em \intop}\nolimits_{#1}^{#2}}%
          {\mathop{\kern 0.1em\vrule width 0.5em height 0.69678ex depth -0.60387ex
                  \kern -0.6em \intop}\nolimits_{#1}^{#2}}}

  \newcommand{\e}{\epsilon}

\def\eqn#1$$#2$${\begin{equation}\label#1#2\end{equation}}
\def\charfn_#1{{\raise1.2pt\hbox{$\chi
_{\kern-1pt\lower3pt\hbox{{$\scriptstyle#1$}}}$}}}

\def\qq1{q_*}
\def\q2{q_{**}}

\newdimen\vintbar
\def\vint{-\kern-\vintbar\int}

\def\C{\mathcal C}

\def\F{\mathcal F}

\def\I{\mathcal I}

\def\P{\mathcal P}

\def\0{\boldsymbol 0}

\newtoks\by
\newtoks\paper
\newtoks\book
\newtoks\jour
\newtoks\yr
\newtoks\pages
\newtoks\vol
\newtoks\publ

\def\name[#1, #2]{#1 #2}
\def\ota{{\hbox{\bf ???}}}
\def\cLear{\by=\ota\paper=\ota\book=\ota\jour=\ota\yr=\ota
\pages=\ota\vol=\ota\publ=\ota}
\def\endpaper{\the\by, \textit{\the\paper},
{\the\jour} \textbf{\the\vol} (\the\yr), \the\pages.\cLear}
\def\endbook{\the\by, \textit{\the\book},
\the\publ, \the\yr.\cLear}
\def\endpap{\the\by, \textit{\the\paper}, \the\jour.\cLear}
\def\endproc{\the\by, \textit{\the\paper}, \the\book, \the\publ,
\the\yr, \the\pages.\cLear}

\begin{document}


\title[Reflected BSDE in Time-dependent Domains]{Reflected BSDE of Wiener-Poisson type in Time-dependent Domains}

\address{Kaj Nystr\"{o}m\\Department of Mathematics, Uppsala University\\
S-751 06 Uppsala, Sweden}
\email{kaj.nystrom@math.uu.se}
\address{Marcus Olofsson\\Department of Mathematics, Uppsala University\\
S-751 06 Uppsala, Sweden}
\email{marcus.olofsson@math.uu.se}

\author{K. Nystr{\"o}m,
 M. Olofsson}

\maketitle

    \begin{abstract}
    \noindent
    In this paper we study multi-dimensional reflected backward stochastic differential equations driven by Wiener-Poisson type processes. We prove existence and uniqueness of solutions,  with reflection in the inward spatial normal direction, in the setting of certain time-dependent domains.\\

    \noindent
    2000  {\em Mathematics Subject Classification.}  \\

    \end{abstract}

\newpage

\setcounter{equation}{0} \setcounter{theorem}{0}
\section{Introduction}
\noindent
Backward stochastic differential equations, BSDEs for short, is by now an established field of research.  The solution to a classical BSDE,  driven by a Wiener process $W$,
is  a pair of processes $(Y,Z)$ such that
\begin{eqnarray*}
Y_t= \xi + \int_t ^T f(s,Y_s,Z_s) ds - \int_t ^T Z_s dW_s,\ 0\leq t\leq T,
\end{eqnarray*}
where $\xi$ is a random variable that becomes known, with certainty, only at time $T$. In this setting $Y_t\in\mathbb R^d$, $d\geq 1$, and in the following we refer to the case $d=1$ as the one-dimensional case and to the case $d>1$ as the multi-dimensional case. Classical BSDEs have turned out important in  many areas of mathematics including mathematical finance, see \cite{EPQ} and the long list of references therein, stochastic control theory and stochastic game
theory, see, e.g., \cite{CK} and \cite{HL}, as well as in the connection to partial differential equations, see, e.g., \cite{BBP} and \cite{PP}.

In \cite{EKPPQ} a notion of \textit{reflected} BSDE was introduced. A solution to a one-dimensional reflected BSDE is a triple of processes $(Y,Z,\Lambda)$ satisfying
\begin{eqnarray*}
Y_t&=& \xi + \int_t ^T f(s,Y_s,Z_s) ds +\Lambda_T-\Lambda_t- \int_t ^T Z_s dW_s,\ 0\leq t\leq T,\notag\\
Y_t&\geq& S_t,
\end{eqnarray*}
where the barrier $S$ is a given (one-dimensional) stochastic process. $\Lambda$ is a continuous increasing process, with $\Lambda_0=0$, pushing the process $Y$ upwards
in order to keep it above the barrier. This is done with minimal energy in the sense that
\begin{eqnarray*}
\int_0^T (S_t-Y_t)d\Lambda_t=0,
\end{eqnarray*}
and consequently $\Lambda$ increases only when $Y$ is at the boundary of the space-time domain $\{(t,s): s>S_t\}$. This type of reflected BSDE has important applications in the context of American options, optimal stopping and obstacle problem, see \cite{EKPPQ},  as well as in the context of stochastic game problems, see \cite{CK}.

In the multi-dimensional case there are at least two different types of reflected BSDEs studied in the literature. 

\noindent The first type of
 multi-dimensional reflected BSDE was first studied in \cite{GP} where the authors considered reflected BSDEs of the form
\begin{eqnarray}\label{Kn4}
Y_t&=& \xi + \int_t ^T f(s,Y_s,Z_s) ds +\Lambda_T-\Lambda_t- \int_t ^T Z_s dW_s,\ 0\leq t\leq T,\notag\\
Y_t&\in & \Omega, \ 0\leq t\leq T,
\end{eqnarray}
where $\Omega\subset\mathbb R^d$. In this case $\Lambda_0=0$ and
\begin{eqnarray}\label{Kn4+}
&&\Lambda_t=\int_{0}^{t}\gamma _{s}d\left\vert \Lambda \right\vert
_{s},\ \gamma _{s}\in N^{1}\left( Y_{s}\right),\notag\\
&& d\left\vert \Lambda \right\vert \left( \left\{ t\in \left[ 0,T\right]
:Y_{t} \in \Omega\right\} \right)=0,
\end{eqnarray}
where $N^{1}\left( Y_{s}\right)$ is the unit inner normal to $\Omega$ at $Y_{s}$.
In particular, the process $\Lambda_t$ is of bounded total variation $|\Lambda|$ and it increases only when $Y$ is at the boundary of $\Omega$. To be more precise, when $Y$ is at the boundary it is pushed into the domain along $\gamma\in N^1(Y)$. In \cite{GP} existence and uniqueness for this problem
is established and we note that this problem, and its analysis, is inspired by and resemble the corresponding theory for reflected stochastic differential equations, see \cite{T}, \cite{S}, and \cite{LS}. Naturally one can attempt, as in the case of reflected SDEs, to study this problem with oblique reflection instead of reflection in the direction of the inner unit normal. However, to the best of our knowledge the case of oblique reflection is a less developed area of research in the context of BSDEs and we are only aware of the work in \cite{R}, where the author studies an obliquely reflected BSDE in an orthant.

\noindent
The second type of multi-dimensional reflected BSDEs occurs in the study of optimal switching problems and stochastic games, see, e.g., \cite{AF}, \cite{AH}, \cite{DHP}, \cite{HT}, \cite{HZ}, and references therein. In the generic optimal switching problem a production facility is considered and it is assumed that the production can run in $d \geq 2$ different production modes. Furthermore, there is a stochastic process $X=(X_t)_{t\geq 0}$ which stands for the market price of the underlying commodities and other financial parameters that influence the production. When the facility is in mode $i$, the revenue per unit time is $f_i(t,X_t)$ and the cost of switching from mode $i$ to mode $j$, at time $t$, is $c_{ij}(t,X_t)$. Let $(Y_t^1,\dots,Y_t^d)$ be the value function associated with the optimal switching problem, on the time interval $[t,T]$, i.e., $Y_t^i$ stands for the optimal expected profit if, at time $t$,
the production is in mode $i$. In this case, one can prove, under various assumptions, see \cite{AF}, \cite{AH}, and \cite{DHP}, that
$(Y_t^1,\dots,Y_t^d)$ solves the reflected BSDE
\begin{align}\label{eq6}
&Y^i_t=\xi_i+\int_t^Tf_i(s, X_s)ds-\int_t^TZ_s^idW_s+\Lambda_T^i-\Lambda_t^i,\notag\\
&Y^i_t\geq \max_{j\in A_i} \left (Y^j_t - c_{ij}(t,X_t)\right), \notag \\ 
&\int_0^T \left ( Y^i_t-\max_{j\in A_i} \left (Y^j_t -c_{ij}(t,X_t) \right ) \right )d\Lambda_t^i=0,
\end{align}
where $i\in\{1,\dots,d\}$, $0\leq t\leq T,$ and $A_i=\{1,\dots,d\}\setminus\{i\}$. In this case the reflected BSDE evolves in the closure
of the time-dependent domain
\begin{eqnarray}\label{eq6+}
D&=&\{(t,y)=(t, y_1,\dots,y_d)\in\mathbb R^{d+1}:\ 0\leq t\leq T,\notag\\
&&y_i \geq \max_{j\in A_i} \left ( y_j - c_{ij}(t,X_t) \right ), \mbox{ for all }i\in\{1,\dots,d\}\}.
\end{eqnarray}
On the boundary of $D$ a reflection occurs and in \cite{HT} the authors refer to this as an oblique reflection. While this oblique reflection seems to have no clear relation to what is referred to as an oblique reflection in the context of \eqref{Kn4}, \eqref{Kn4+},
it is still fair to refer to the problem in \eqref{eq6} as an obliquely reflected BSDE. However, we emphasize that the problems in \eqref{Kn4}, \eqref{Kn4+} and \eqref{eq6} are significantly different.

In this paper we consider the problem in \eqref{Kn4}, \eqref{Kn4+} in time-dependent domains and with underlying stochastic processes beyond Brownian motion. In light of \eqref{eq6}, \eqref{eq6+}, and corresponding developments for
reflected SDEs, see \cite{C}, \cite{CGK}, \cite{LS}, \cite{NO}, \cite{S}, and \cite{T}, it is natural to allow for time-dependent domains and in many cases this extra feature calls for additional arguments in comparison with the case of time-independent domains. In particular, we here consider \eqref{Kn4}, \eqref{Kn4+} in the context of time-dependent domains, and along the lines of \cite{GP}. In addition, we allow the BSDE to be driven by a Wiener-Poisson type process and our main result is a generalization of \cite{GP} and \cite{O} to a time-dependent setting. In general, it seems difficult to generalize \cite{GP} and its proofs beyond the assumption of convexity of the time-slices of the domain. Indeed, the assumption on convexity is heavily explored in \cite{GP} and \cite{O}. Beyond ensuring the existence of projections, convexity establishes the positivity of certain terms appearing when applying the Ito formula. In this sense, one may say that the arguments are slightly rigid as the structural assumption of convexity seems crucial. In our analysis,
it turns out that we are only able to pull the arguments of \cite{GP} through in the context of time-dependent domain having a similar rigidity in time. More precisely, in our case the time slices must be non-increasing and hence the domain must be non-expanding as a function of time. Under such a structural assumption though, we are able to generalize \cite{GP} and \cite{O} to a time-dependent setting. Finally, we note that it is an interesting open problem to understand if, in analogy with the connection between optimal switching problems and the problem in \eqref{eq6}, the problem in \eqref{Kn4}, \eqref{Kn4+} can be naturally associated to some stochastic optimization problem.

\setcounter{equation}{0} \setcounter{theorem}{0}
\section{Statement of main result}
\noindent
In this section we state our main result. To do this properly we first briefly discuss the geometry and processes of Wiener-Poisson type and define the reflected BSDE studied in this paper.

 \subsection{Geometry}  

Given $d\geq 1$, we let $\left\langle \cdot ,\cdot
\right\rangle $ denote the standard inner product on $%
\mathbb{R}^{d}$ and $\left\vert z\right\vert =\left\langle z,z\right\rangle
^{1/2}$ be the Euclidean norm of $z\in\mathbb R^d.$ Whenever $z\in\mathbb{R}^{d}$ and $r>0$, we let $B_r(z)$ and $S_r(z)$ denote the ball and sphere of radius $r$, centered at $z$, respectively, i.e. $B_{r}\left( z\right) =\left\{ y\in\mathbb{R}^{d}:\left\vert z-y\right\vert <r\right\} $ and $S_{r}\left( z\right)=\left\{ y\in\mathbb{R}^{d}:\left\vert z-y\right\vert =r\right\}$. Moreover, given $F\subset\mathbb{R}^{d}$, $E\subset\mathbb{R}
^{d}$, we let $\bar{F}$, $\bar{E}$ be the closure of $F$ and $E$,
respectively, and we let $d\left( y,E\right) $ denote the Euclidean distance
from 
$y\in\mathbb{R}^{d}$ to $E$.  Given $d\geq 1$, $T>0$ and an open, connected set $D^{\prime}\subset\mathbb{R}
^{d+1}$ 
we will refer to
\begin{equation*}
D=D^{\prime }\cap ([0,T]\times \mathbb{R}^{d}),  
\end{equation*}%
as a time-dependent domain. 

Given $D$ and $t\in \left[ 0,T\right] $, we
define the time sections of $D$ as 
\begin{equation}\label{eq:timeslice}
D_{t}=\left\{ z:\left( t,z\right) \in D\right\}.
\end{equation} We assume that
\begin{equation}\label{timedep+}
D_{t}\neq \emptyset, D_t\text{ is open, bounded and connected for every }t\in \left[ 0,T\right], 
\end{equation}%
and that
\begin{equation}
D_{t}\text{ is convex
for every }t\in \left[ 0,T\right].  \label{timedep+1}
\end{equation}%
Furthermore, following \cite%
{CGK}, we let
\begin{equation*}
l\left( r\right) =\sup_{\substack{ s,t\in \lbrack 0,T]  \\ \left\vert s-t\right\vert \leq r}}\,\sup_{z\in \overline{D_{s}}}d\left( z,D_{t}\right) ,
\end{equation*}
be the modulus of continuity of the variation of $D$ in time and we assume that
\begin{equation}\label{limitzero}
\lim_{r \to 0^{+}}l\left( r\right) =0.
\end{equation}
We also assume that
\begin{equation}\label{timedep+2}
D_{t'}\subseteq D_t\text{ whenever $t'\geq t$, }t', t\in \left[ 0,T\right]. 
\end{equation}%
Note that \eqref{timedep+2} implies that
\begin{equation*}
l\left( r\right) =\sup_{\substack{ t\in \lbrack 0,T], [t-r,t+r]\in [0,T]}}\sup_{z\in \overline{D_{t-r}}}d\left( z,D_{t+r}\right).
\end{equation*}%
We let $\partial D$ and $\partial D_{t}$, for $t\in \left[ 0,T\right] $,
denote the boundaries of $D$ and $D_{t}$, respectively, and we let $N_{t}\left(
z\right) $ denote the cone of inward normal vectors at $z\in \partial D_{t}$%
, $t\in \lbrack 0,T]$. Note that it follows from \eqref{timedep+1} that $N_{t}\left(
z\right)\neq\emptyset $ for every $z\in \partial D_{t}$, $t\in \lbrack 0,T]$. In general, the cone $N_{t}\left( z\right) $ of inward normal vectors at $z\in \partial
D_{t}$, $t\in \lbrack 0,T]$, is defined as being equal to the set consisting
of the union of the set $\left\{ 0\right\}$ and the set
\begin{equation*}
\left\{ v\in\mathbb{R}^{d}:v\neq 0,\exists \rho >0\text{ such that }B_{\rho }\left( z-\rho
v/\left\vert v\right\vert \right) \subset \mathbb{R}^{d}\setminus D_t\right\} .  
\end{equation*}%
Note that this definition does not rule out the possibility of several unit
inward normal vectors at the same boundary point. Given $N_{t}\left(
z\right) $, we let $N_{t}^{1}(z):=N_{t}(z)\cap S_{1}(0)$, so that $%
N_{t}^{1}(z)$ contains the set of vectors in $N_{t}(z)$ with unit length. In this paper we consider reflected BSDEs in the setting of time-dependent domains $D$ satisfying
\eqref{timedep+}-\eqref{timedep+2}. Furthermore, reflection at $z\in \partial D_{t}$, $t\in \lbrack 0,T]$, is considered in the direction of a
  unit spatial inward normal in the cone $N_{t}\left(
z\right)$.

 \subsection{Processes of Wiener-Poisson type}
Throughout the paper we let  
$$\left(\Omega ,\mathcal{F} ,\{ \mathcal{F}_t\}_{t \in [0,T]}, \mathbb{P}\right)$$ be a complete Wiener-Poisson space in $\mathbb R^n\times \mathbb R^m\setminus\{0\}$ with Levy measure $\lambda$. In particular, $\left(\Omega ,\mathcal{F} ,\mathbb P \right) $ is a complete probability space and $\{\mathcal{F}_t\},_{t\in[0,T]}$ is an increasing, right continuous family of complete sub $\sigma$-algebras of $\mathcal{F}$. We let $(W_t,\{F_t\})_{t \in [0,T]}$ be a standard Wiener process in $\R^n$ and $(\mu_t, \{\F_t\})_{ t\in [0,T]}$ be a martingale measure in $\R^m \setminus\{0\}$, which is assumed to be independent of $W$, and which corresponds to a standard Poisson random measure $p(t,A)$. Indeed, for any Borel measurable subset $A$ of $\R^m \setminus \{0\}$ such that the Levy measure $\lambda$ satisfies $\lambda(A) <+ \infty $, we have
\begin{equation*}
\mu_t(A) = p(t,A) - t \lambda(A)
\end{equation*}
where $p(t,A)$ satisfies
\begin{equation*}
E[ p(t,A) ] = t \lambda(A).
\end{equation*}
We let $U:=\R^m \setminus \{0\}$ and we let $\mathcal{U}$ be its Borel $\sigma$-algebra. We assume that $\{\F_t\}_{t\in [0,T]}$ is the filtration generated by $W_t$ and the jump process corresponding to the Poission random measure $p$, augmented with the $\mathbb P$-null sets of $\F$, i.e.,
\begin{equation*}
\F_t=\sigma\left( \int \int _{A \times [0,s]} p(ds,dx): s\leq t, A\in \mathcal{U} \right ) \vee  \sigma \left (W_s, s\leq t \right) \vee \F_0,
\end{equation*}
where $\F_0$ denotes the $\P$-null sets of $\F$ and $\sigma_1 \vee \sigma_2 $ denotes, given two $\sigma$-algebras $\sigma_1$ and $\sigma_2$, the $\si$-algebra generated by $\si_1 \cup \si_2$.

\subsection{Reflected BSDEs}  Given $T>0$, we let $\mathcal{D}\left( \left[ 0,T\right] ,
\mathbb{R}^{d}\right) $ denote the set of c\`{a}dl\`{a}g functions $v(t)=v_{t}:\left[ 0,T\right] \rightarrow
\mathbb{R}^{d}$, 
i.e., functions which are right continuous and have left limits. We denote the set of functions $w(t)=w_{t}:\left[ 0,T\right]
\rightarrow \mathbb{R} ^{d}$ with bounded variation by $\mathcal{BV}\left( \left[ 0,T\right] , \mathbb{R}^{d}\right) $ and we let $\left\vert w \right\vert $ denote the total
variation of $w \in \mathcal{BV}\left( \left[ 0,T\right] ,\mathbb{R}^{d}\right) $.
Recall that the total variation process $\left\vert w \right\vert $ is defined as
$$
\left\vert w \right\vert_t=\sup\sum_{k=1}^{n}| w_{t_i}-w_{t_{i-1}}|,\ 0\leq t\leq T,
$$
where the supremum is taken over all finite partitions $0=t_0<t_1<\dots<t_n=t$. Furthermore, we have that
\begin{eqnarray}\label{bff}
w_t=\int_0^t\nu_sd\left\vert w \right\vert_s
\end{eqnarray}
where $\nu_s$ is a vector of length 1, i.e., $|\nu_s|=1$ for $\left\vert w \right\vert$-almost all $s$. Let $$\left(
\Omega ,\mathcal{F}, \{\mathcal{F}_t\}, \mathbb{P}, W_t,\mu_t, t\in[0,T] \right)$$ be the complete Wiener-Poisson space in $\mathbb R^n\times \mathbb R^m\setminus\{0\}$, with Levy measure $\lambda$, as outlined above. Let $L^2(\Omega, \mathcal{F}_T, \mathbb P)$ be the space of square integrable, $\mathcal{F}_T$-adapted random variables and let $L^2(U, \mathcal{U},  \lambda; \R^d)$ be the space of functions which are $\mathcal{U}$-measurable, maps values in $U$ to $\R^d$, and which are square integrable on $U$ with respect to the Levy-measure $\lambda$. 
In the following we let
the norm 
$$\|z\|:=(\sum_{i,j}|z_{ij}|^2)^{1/2}$$
be defined on real-valued $(d\times n)$-dimensional matrices and we define the norm 
$$\| u(e) \| := \left (\int_V |u(e)|^2 \lambda(de) \right )^{1/2}$$
on $L^2(U, \mathcal{U},  \lambda; \R^d)$. Let $\xi=(\xi_1,\dots,\xi_d)$ be such that
\begin{eqnarray}\label{data}
\mbox{$\xi \in L^2(\Omega, \mathcal{F}_T, \mathbb P)$ and $\xi\in D_T$ a.s.}
\end{eqnarray}
Let $f: \Omega \times  [0,T] \times \R^d \times\R^{d\times n}\times L^2(U, \mathcal{U},  \lambda; \R^d) \to \R^d$ be a function such that
\begin{eqnarray}\label{data++}
(i)&&\mbox{$(\omega,t)\to f(\omega,t,y,z,u) $ is $\mathcal{F}_t$ progressively measurable whenever}\notag\\
 &&\mbox{$(y,z,u)\in \R^d \times\R^{d\times n}\times L^2(U, \mathcal{U},  \lambda; \R^d)$},\notag \\
(ii)&&\mbox{$E \bigl [\int_0 ^T |f(\omega,t,0,0,0)|^2 dt\bigr ] < \infty $},\notag\\
(iii)&&\mbox{$|f(\omega, t, y, z, u)-f(\omega, t, y', z', u')|\leq c(|y-y'|+\|z-z'\|+\|u-u'\|)$}\notag\\
&&\mbox{for some constant $c$ whenever}\notag\\
 &&\mbox{$(y,z,u), (y',z',u') \in \R^d \times \R^{d\times n} \times L^2(U, \mathcal{U},  \lambda; \R^d)$, $(\omega,t)\in \Omega \times  [0,T]$}.
\end{eqnarray}
In the context of BSDEs, $\xi$ and $f$ are usually referred to as terminal value and driver of the BSDE, respectively. We are now ready to formulate the notion of reflected BSDE considered in this paper.

 \begin{definition}\label{rbsde} Let $d\geq 1$ and $T>0$. Let $D\subset\mathbb{R}^{d+1}$ be a time-dependent domain satisfying \eqref{timedep+}. Given $(\xi,f)$ as in \eqref{data}-\eqref{data++}, a quadruple $(Y_t, Z_t, U_t, \Lambda_t)$ of progressively measurable processes with values in $\R^d \times \R^{d \times m} \times L^2(U, \mathcal{U}, \lambda; \R ^d) \times \R^d$
is said to be a solution to a reflected BSDE, with reflection in the inward spatial normal direction, in $D$, and with data $(\xi,f),$ if the following holds. $Y\in \mathcal{D}\left( \left[ 0,T\right] ,
\mathbb{R}^{d}\right) $, $Z$ and $U$ are predictable processes, and
\begin{eqnarray*}
(i)&&E\left [\sup _{0 \leq t \leq T} |Y_t|^2\right ] <\infty, \\
(ii)&& E \left [\int_0 ^T \|Z_t\|^2 dt +  \int _0^T \int _U|U_s(e)|^2 \lambda(de)ds \right ] < \infty,\notag\\
(iii)&&Y_t= \xi + \int_t ^T f(s,Y_s,Z_s,U_s) ds + \Lambda_T- \Lambda_t\notag\\
 &&- \int_t ^T Z_s dW_s - \int _t^T \int _U U_s(e) \mu(de,ds) \quad \mbox{a.s.}, \\
(iv)&& Y_t \in\overline{D_t} \quad \mbox{a.s.},
\end{eqnarray*}
 whenever $t \in [0, T]$. Furthermore, $\Lambda\in \mathcal{BV}\left( \left[ 0,T\right] ,%
\mathbb{R}
^{d}\right) $ and
\begin{eqnarray*}
(v)&&\Lambda_t=\int_{0}^{t^{}}\gamma _{s}d\left\vert \Lambda \right\vert
_{s},\ \gamma _{s}\in N_{s}^{1}\left( Y_{s}\right) \mbox{ whenever } Y_s \in \partial D_{s},\\
(vi)&& d\left\vert \Lambda \right\vert \left( \left\{ t\in \left[ 0,T\right]
:\left(t,Y_{t}\right) \in D\right\} \right)=0.
\end{eqnarray*}
\end{definition}

\subsection{Statement of the main result} Concerning reflected BSDEs we establish the following result.
\begin{theorem}\label{maint1} Let $d\geq 1$ and $T>0$. Let $D\subset \mathbb{R}^{d+1}$ be a time-dependent domain satisfying \eqref{timedep+}-\eqref{timedep+2} and assume the terminal data $\xi$ and driver $f$ satisfy \eqref{data}-\eqref{data++}. Then there exists a unique solution $(Y_t, Z_t, U_t, \Lambda_t)$ to the reflected BSDE, with reflection in the inward spatial normal direction, in $D$, and with data $(\xi,f),$  in the sense of Definition \ref{rbsde}.
\end{theorem}

\subsection{Organization of the paper} The rest of the paper is organized as follows. Section \ref{sec:prel} is of preliminary nature and we here focus on the geometry of the time-dependent domain as well as smooth approximations of it. We also recall the Ito formula in the context of Wiener-Poisson processes. In section \ref{sec:lemmas} we introduce a sequence of non-reflected BSDEs, constructed by penalization techniques, and develop a number of technical lemmas for these. Finally, using the results of section \ref{sec:lemmas}, we prove the main result in section \ref{sec:proof}.

\section{Preliminaries} \label{sec:prel}
\noindent
 Let $d\geq 1$ and $T>0$. Let $D\subset\mathbb{R}^{d+1}$ be a time-dependent domain satisfying \eqref{timedep+}~-~\eqref{timedep+2}. Let $N =N_{t}(z)=N (t,z)$ denote the cone of inward normal vectors given for all $z\in \partial D_{t}$, $t\in \lbrack 0,T]$. Note that by \eqref{timedep+1} there exists, for any $y\in\mathbb{R}^{d}\setminus \overline{D}_{t}$, $t\in \lbrack 0,T]$,
at least one projection of $y$ onto $\partial D_{t}$ along $N_{t}$,
denoted $\pi\left( t,y\right) $, which
satisfies%
\begin{equation*}
\left\vert y-\pi\left(t, y\right) \right\vert=d\left( y,D_{t}\right).  
\end{equation*}
To have $\pi(t,\cdot)$ well-defined for all $y\in\mathbb R^d$ we also let $\pi\left(t, y\right)=y$ whenever $y\in\overline{D_t}$. The following lemma summarizes a few standard results from convex analysis.
\begin{lemma}\label{lemmaa} Let $D\subset
\mathbb{R}^{d+1}$ be a time-dependent domain satisfying \eqref{timedep+} and assume \eqref{timedep+1} and \eqref{timedep+2}. Then the following holds whenever $t\in[0,T]$:
\begin{eqnarray*}
(i)&&\langle y^\prime-y,y-\pi\left(t, y\right)\rangle \leq 0,\mbox{ for } (y,y^\prime) \in \mathbb R^d\times\overline{D_t}, \mbox{ and }\notag \\
(ii)&&\langle y^\prime-y, y-\pi(t,y)\rangle \leq \langle y^\prime-\pi(t,y^\prime),y-\pi(t,y)\rangle,\notag \\
(iii)&& |\pi(t,y) - \pi(t,y')| \leq |y-y'|,  
\end{eqnarray*}
whenever $(y,y^\prime) \in \mathbb R^d\times\mathbb R^d$. Furthermore, there exists $P_T\in D_T$ and $\gamma$, $1\leq\gamma<\infty$, depending on $d(P_T, \partial D_T)$, such that
\begin{eqnarray*}
(iv)&& \langle y-P_T, y-\pi(t,y)\rangle \geq \gamma^{-1} |y-\pi(t,y)|,   \mbox{ for any } y\in \mathbb R^d, t\in[0,T].
\end{eqnarray*}
\end{lemma}

\subsection{Geometry of time-dependent domains - smooth approximations}
Note that the assumptions in \eqref{timedep+}, \eqref{timedep+1}, and \eqref{timedep+2} contain no particular smoothness assumption. Instead, we will in the following construct smooth approximations of $D$ to enable the use of Ito's formula.  In the following we let $h(t,y)=d(y,D_t)$ whenever $y\in\mathbb R^d$, $t\in [0,T]$, with the convention that $h(t,y)=0$ if $y\in\overline{D_t}$. Assuming that $D_T\neq\emptyset\neq D_0$ we
let $h(t,y)=h(T,y)$ whenever $y\in\mathbb R^d$, $t>T$, and $h(t,y)=h(0,y)$ whenever $y\in\mathbb R^d$, $t<0$. Using this notation we see that
 $$\overline D=\{(t,x)\in [0,T]\times\mathbb R^d|\ h(t,x)=0\}.$$ Let
$\phi=\phi(s,y)$ be a smooth mollifier in $\mathbb R^{d+1}$, i.e., $\phi\in C_0^\infty(\mathbb R^{d+1})$, $0\leq\phi\leq 1$, the support
 of $\phi$ is contained in the Euclidean unit ball in $\mathbb R^{d+1}$, centered at $0$, and $\int\phi dyds=1$. Let, for $\delta>0$ small,
$\phi_\delta(s,y)=\delta^{n+1}\phi(\delta^{-1}s,\delta^{-1}y)$. Based on $\phi_\delta$  we let, whenever $(t,y)\in\mathbb R\times\mathbb R^d$,
\begin{eqnarray*}
h_{\delta}(t,y)=(\phi_{\delta}\ast h)(t,y)=\int_{\mathbb R}\int_{\mathbb R^n}\phi_{\delta}(t-s,y-x)h(s,x)dxds
\end{eqnarray*}
be a smooth mollification of $h$. Furthermore, we let
\[
 h ( F, G ) = \max  (  \sup   \{ d ( y, F ) : y \in G \}, \sup \{ d ( y, G ) : y \in F \} ) 
\] 
denote the Hausdorff distance
between the sets $ F, G \subset \mathbb R^d$. Based on $h_{\delta}$ we introduce a smooth approximation of $D$ as follows. Given $\eta$ fixed and $\delta>0$, we let
\begin{equation*}
D^{\eta}_{\delta} = \{(t,x)\in [0,T]\times\mathbb R^d|h_\delta(t,x)< \eta\}.
\end{equation*}
Then $D^{\eta}_{\delta}$ converges to $D^{\eta}:= \{(t,x)\in [0,T]\times\mathbb R^d|h(t,x)< \eta\}$ in the Hausdorff distance sense as $\delta\to 0$. Note that $D^\eta_\delta$ is a $C^\infty$-smooth domain. Hence, letting $\delta, \eta \to 0$ we have the following lemma.
\begin{lemma}\label{lemmauu1-} Let $D\subset
\mathbb{R}
^{d+1}$ be a time-dependent domain satisfying \eqref{timedep+}~-~\eqref{timedep+2}. Then, for any
$\epsilon>0$ there exists a time-dependent domain $D_\epsilon\subset \mathbb{R} ^{d+1}$ satisfying \eqref{timedep+}~-~\eqref{timedep+2} such that $D_\epsilon$ is $C^\infty$-smooth and 
\begin{eqnarray*}
h(D_{t},D_{\epsilon,t})<\epsilon\mbox{ for all $t\in[0,T]$},
\end{eqnarray*}
where $D_t$ is as defined in \eqref{eq:timeslice}, and $D_{\e,t}=\{x: (x,t) \in D_\e\}$.
\end{lemma}

 Let, for all $t\in \lbrack 0,T]$, $N_\epsilon =N_{\epsilon,t}(z)=N_\epsilon(t,z)$ denote the cone of inward normal vectors at $z \in \partial D_{\e,t}$. Due to the smoothness of $\partial D_{\epsilon,t}$, $N_\epsilon(t,z)$ consists of a single vector. For any $y\in
\mathbb{R}
^{d}\setminus \overline{D}_{\e,t}$, $t\in \lbrack 0,T]$, we let
$\pi_\epsilon\left( t,y\right) $ denote the projection of $y$ onto $\partial D_{\epsilon,t}$ along this unique direction. To have $\pi_\epsilon(t,\cdot)$ well-defined for all $y\in\mathbb R^d$ we also let $\pi_\epsilon\left(t, y\right)=y$ whenever $y\in\overline{D}_{\epsilon,t}$.
In this setting, the following lemma  can be proven as Lemma 2.2  in \cite{GP} as we are only considering fixed time slices $D_t$ of $D$.
\begin{lemma}
\label{lemmauu1}
Let $D\subset\mathbb{R}^{d+1}$ be a time-dependent domain satisfying \eqref{timedep+}~-~\eqref{timedep+2} and let $D_{\epsilon}$
be constructed as above. There exists a constant $c$ such that, if $\epsilon\in (0,1)$, $y\in\mathbb R^d$ and $t\in[0,T]$, then
\begin{eqnarray*}
(i)&&|\pi(t,y)-\pi_\epsilon(t,y)| \leq c \sqrt{\epsilon^2+\epsilon d(y, D_{\epsilon,t})},\\
(ii)&&|\pi(t,y)-\pi_\epsilon(t,y)| \leq c \sqrt{\epsilon^2+\epsilon d(y, D_t)}.
\end{eqnarray*}
\end{lemma}

The following lemma is a corollary of Lemma \ref{lemmauu1}.
\begin{lemma} \label{lemmauu2} Let $D$ and $D_{\epsilon}$ be as in the statement of Lemma \ref{lemmauu1}. There exists a constant $c$ such that, if
$\epsilon\in (0,1)$ and $y\in\mathbb R^d$, $t\in[0,T]$, then
\begin{eqnarray*}
(i)&&|\pi(t,y)-\pi_\epsilon(t,y)| \leq c \sqrt{\epsilon}(1+d(y, D_{\epsilon,t})),\\
(ii)&&|\pi(t,y)-\pi_\epsilon(t,y)|\leq c \sqrt{\epsilon}\sqrt{d(y, D_{t,\epsilon})}\mbox{ whenever } d(y,D_{\epsilon,t})>\epsilon.
\end{eqnarray*}
\end{lemma}

We here also recall Ito's formula in the context of Wiener-Poisson processes, see \cite{OS}. Here and in the following, we denote by $\C^{1,2}([0,T]\times \R^d, \R)$ the space of functions $\varphi(t,y) : [0,T]\times \R^d \to \R$ which are once continuosly differentiable with respect to $t\in [0,T]$ and twice continuosly differentiable with respect to $y \in \R^d$  and we let $A^\ast$ denote the transpose of the matrix $A$.

\begin{lemma}\label{lemmauu2+}
Let $Y_t$ be a Levy process such that
\begin{equation*}
dY_t = f_t dt + \sigma _t dW_t + \int _U U_t(e) \mu (de,dt)
\end{equation*}
and let $\varphi(t,y) \in \C^{1,2}([0,T]\times \R^d, \R)$. Then
\begin{eqnarray*}
d\varphi(t, Y_t) &=& \partial_t\varphi(t, Y_{t}) dt+ (\nabla \varphi(t, Y_{t^-})) \left [  f_t dt + \sigma_t dW_t +  \int _U U_s(e) \mu (de,ds) \right] \notag \\
&+& \sum_{i,j}\frac{1}{2}  (\sigma_t\sigma_t^\ast)_{ij}\partial^2_{y_iy_j}\varphi_\epsilon(t, Y_{t}) dt \notag \\
&+& \int_U \left [\varphi(t,Y_{t^-} + U_t(e))-\varphi(t,Y_{t^-}) - \langle \nabla \varphi(t,Y_{t^-}), U_t(e) \rangle \right] p(de,dt).
\end{eqnarray*}
\end{lemma}

Based on the smooth domain $D_\e$ we define the function $\varphi_\epsilon(t,y):= (d(y,D_{\epsilon,t}))^2 = |y-\pi_\epsilon(t,y)|^2$ to which Ito's formula needs to be applied in the proof of Theorem \ref{maint1},. Note that although $D_\e$ is a smooth domain, the second (spatial) derivative of $\varphi_\e$ is not continuous at the boundary of $D_{\e,t}$ and thus Lemma \ref{lemmauu2+} is not directly applicable. To enable the use of Ito's formula we therefore proceed along the lines of \cite{LS}, see also \cite{GP}, and extend our distance function $\varphi_\e$ across the boundary and into the domain $D_{\e,t}$. Indeed, since $\partial D_{\e,t}$ is smooth there exists a neighbourhood $V_{\e,t}$ of $\partial D_{\e,t}$ such that, for $y \in D_{\e,t} \cap V_{\e,t}$, there exists a unique pair $(x, s) \in \partial D_{\e,t} \times \R^+$ such that $y=x+s\gamma$, where $\gamma \in N^1_{\e,t}(x)$. Recall that $N^1_{\e,t}(x)$, the cone of unit inward normal vectors to $D_{\e,t}$, at $x \in \partial D_{\e,t}$, contains only a single vector.
By the convexity of $D_{\epsilon,t}$ we also have
$$y=x+s \gamma, \ \mbox{ for } x=\pi_\epsilon(t,y),\ s=-d(y, D_{\epsilon,t}), \ \gamma \in N^1_{\e,t}(\pi_\e(t,x)),$$
whenever $y\in\mathbb R^d\setminus\overline{D_{\epsilon,t}}$. Hence, for $t \in [0,T]$ fixed, we can define a smooth map $\phi_\epsilon : \R^d \to\R$ such that
\begin{eqnarray*}
\phi_\epsilon(t,y)&=&s\quad \mbox{ when } y\in (\mathbb R^d\setminus\overline{D_{\epsilon,t}})\cup (
V_{\epsilon,t}\cap \overline{D_{\epsilon,t}}),\notag\\
\phi_\epsilon(t,y)&>&0\quad \mbox{ otherwise. }
\end{eqnarray*}
Using such a function $\phi_\epsilon$ we have
\noindent \begin{eqnarray*}
D_\epsilon&=&\{(t,y):\ t\in [0,T], y\in\mathbb R^d,\ \phi_\epsilon(t,y)>0\},\notag\\
\partial D_{\epsilon,t}&=&\{y\in\mathbb R^d,\ \phi_\epsilon(t,y)=0\},\mbox{ for }t\in [0,T],\notag\\
(\mathbb R^d\times [0,T])\setminus\overline{D_\epsilon}&=&\{(t,y):\ t\in [0,T], y\in\mathbb R^d,\ \phi_\epsilon(t,y)<0\}.
\end{eqnarray*}
Note that $\phi_\e(t,y)$ is smooth also across the boundary of $D_{\e,t}$ and that $\varphi_\e(t,y)=(\phi_\e(t,y)^-)^2$. Following \cite{GP}, we can now take an approximating sequence of smooth functions $\{g_n\}_{n\geq0}$, tending to $g(x)=(x^-)^2$ as $n \to \infty$,
and construct a sequence of smooth functions $\{\varphi^n_\e(t,y)=g_n(\phi_\e(t,y))\}_{n\geq 0}$ such that Ito's formula can be applied to $\varphi^n_\e$ for every $n\geq 0$ and such that
$\varphi^n_\e(t,y)$, $\frac{\partial}{\partial t}\varphi^n_\e(t,y)$, $\frac{\partial}{\partial y_i}\varphi^n_\e(t,y)$, $\frac{\partial^2}{\partial y_iy_j}\varphi^n_\e(t,y)$ tend to $\varphi_\e(t,y)$, $\frac{\partial}{\partial t}\varphi_\e(t,y)$, $\frac{\partial}{\partial y_i}\varphi_\e(t,y)$, $\frac{\partial^2}{\partial y_iy_j}\varphi_\e(t,y)$, respectively, as $n \to \infty$. Having such an approximation in mind, we will from here on in slightly abuse notation and apply the Ito formula directly to $\varphi_\e(t,y)$.

Finally, the following lemma is the result of the geometric assumptions on $D$ that we will use in the context of Ito's formula.
\begin{lemma} \label{lemmauu2ny} Let $D$ and $D_{\epsilon}$ be as in the statement of Lemma \ref{lemmauu1} and let $\varphi_\e(t,y)$ be defined as
$$\varphi_\epsilon(t,y):= (d(y,D_{\epsilon,t}))^2 = |y-\pi_\epsilon(t,y)|^2,\ t\in [0,T],\ y\in\mathbb R^d.$$
Then
\begin{eqnarray} \label{lemmauu2+-}
&(i) &\partial_t\varphi_\epsilon(t,y)\geq 0, \notag \\ 
&(ii) &\partial^2_{y_iy_j}\varphi_\epsilon(t,y)\xi_i\xi_j\geq 0,
\end{eqnarray}
whenever $t\in [0,T],\ y, \xi \in\mathbb R^d$, and
\begin{equation} \label{lemmauu2+-+}
\varphi_{\e}(t,y+z) - \varphi_{\e}(t,y) - \langle \nabla \varphi_{\e}(t,y) ,z \rangle\geq 0
\end{equation}
whenever $t\in [0,T],\ y,z\in\mathbb R^d$.
\end{lemma}
\begin{proof}
\eqref{lemmauu2+-} $(i)$ follows from \eqref{timedep+2} and \eqref{lemmauu2+-} $(ii)$ follows from the convexity of $D_{\e,t}$. Finally, Taylors formula and \eqref{lemmauu2+-} $(ii)$ yields \eqref{lemmauu2+-+}.
\end{proof}

\section{Estimates for approximating problems: technical lemmas} \label{sec:lemmas}
\noindent
To prove the existence part of Theorem \ref{maint1} we use the method of penalization. Indeed, for each $n\geq 1$, we construct  a  quadruple $(Y_t^n, Z_t^n, U_t ^n, \Lambda_t^n)$ through penalization and we prove that $(Y_t^n, Z_t^n, U_t ^n, \Lambda_t^n)$ converges, as $n\to \infty$, to a solution $(Y_t, Z_t, U_ t, \Lambda_t)$ of the reflected backward stochastic differential equation, with reflection in the inward spatial normal direction, in $D$, and with data $(\xi,f),$ as defined in Definition \ref{rbsde}. Uniqueness is then proved by an argument based on Ito's formula. In the following we let $(\xi,f)$ be as in \eqref{data}-\eqref{data++}, and we let $D\subset\mathbb{R}^{d+1}$ be a time-dependent domain satisfying \eqref{timedep+}~-~\eqref{timedep+2}. Furthermore, we let $c$ denote a generic constant which may change value from line to line.
\subsection{Construction of $(Y_t^n, Z_t^n, U_t^n, \Lambda_t^n)$}
Let, for any $n\in\mathbb Z_+$,
\begin{eqnarray}\label{ffa1}
f_n(t,y,z,u):=f(t,y,z,u)-n(y-\pi(t,y)).
\end{eqnarray}
Then, for $n$ fixed, $f_n$ satisfies \eqref{data++} since $\pi$ has the Lipschitz property in space, see Lemma \ref{lemmaa} $(iii)$. Hence, using results concerning existence and uniqueness for (unconstrained) BSDEs driven by Wiener-Poisson type processes, see Lemma 2.4 in \cite{TL}, we can conclude that there exist, for each $n\in\mathbb Z_+$, a unique
triple $(Y_t^n, Z_t^n, U_t^n)$ and a constant $c_n$, $1\leq c_n<\infty$, such that

\begin{eqnarray}\label{ffa}(i)&&E \left [\sup_{0 \leq t \leq T} |Y_t^n|^2 \right] \leq c_n,\notag\\
(ii)&& E \left [\int_0 ^T \|Z^n_t
\|^2 dt +  \int _0^T \int _U|U^n_s(e)|^2 \lambda(de)ds \right] < \infty,
\notag\\
(iii)&&Y^n_t= \xi + \int_t ^T f_n(s,Y^n_s,Z^n_s,U^n_s) ds\notag\\
 &&- \int_t ^T Z^n_s dW_s- \int_t^T \int _U U^n_s(e) \mu(de,ds).
\end{eqnarray}
Note also that from \cite{TL} we have $Y^n \in \mathcal{D}\left( \left[ 0,T\right] , \R^{d}\right) $. Given $(Y_t^n, Z_t^n, U_t ^n)$ we define, for $n\in\mathbb Z_+$,  the process $\Lambda_t^n$ through
\begin{equation}\label{ffa3}
 \Lambda_t ^n =-n \int_0^t ( Y_s ^n-\pi(s,Y_s ^n)) ds.
\end{equation}
Note that
\begin{equation*}
 \Lambda_t ^n =\int_0^t \frac{-( Y_s ^n-\pi(s,Y_s ^n))}{| Y_s ^n-\pi(s,Y_s ^n)|}d|\Lambda^n|_s
\end{equation*}
and that $-( Y_s ^n-\pi(s,Y_s ^n))/{| Y_s ^n-\pi(s,Y_s ^n)|}$ is an element in the inwards directed normal cone to $D_s$ at
$\pi(s,Y_s ^n)\in\partial D_s$. Furthermore, using \eqref{ffa1} and \eqref{ffa3} we see that \eqref{ffa} $(iii)$ can be rewritten as
\begin{eqnarray*}
 Y_t ^n&=& \xi +\int_t ^T f(s, Y_s ^n, Z_s ^n, U_s ^n) ds +\Lambda^n_T-\Lambda_t^n\notag\\
 &&-\int _t ^T Z_s^n dW_s - \int _t^T \int _U U_s^n(e) \mu(de,ds),
\end{eqnarray*}
for all $t\in [0,T]$. Recall that $Y^n, \Lambda^n, U^n$, $W_t$ and  $Z^n$ are multi-dimensional processes. In particular, $Y_t^n=(Y_t ^{1,n}, \dots,Y_t ^{d,n})$, $\Lambda_t^n=(\Lambda_t ^{1,n},\dots,\Lambda_t ^{d,n})$, $U_t^n=(U_t ^{1,n},\dots, U_t ^{d,n})$, $W_t=(W_t^1,\dots,W_t^n)$ and $ Z_t^n$ is a $d\times n$-matrix with entries $Z_t^{i,j,n}$ and columns $Z_t^{j,n}$.
\subsection{A priori estimates for $(Y_t^n, Z_t^n, U_t^n, \Lambda_t^n)$}
\begin{lemma}\label{ll1} There exists a constant $c$, $1\leq c<\infty$, independent of $n$, such that
\begin{eqnarray*}
(i)&&E  \left [\sup_{0 \leq t \leq T}|Y_t^n|^2 \right ] \leq  c,\notag\\
(ii)&&E \left [\int _t ^T \|Z_s^n\|^2  + \int _t^T \int _U |U_s^n(e)|^2 \lambda(de)ds \right ]\leq c,\ t\in[0,T],\\
(iii)&& E \left [n\int _t ^T|Y_s^n-\pi(s,Y_s ^n)|ds \right ]\leq c, \ t\in[0,T].
\end{eqnarray*}
\end{lemma}
\begin{proof} Let $P_T\in D_T$ be as in Lemma \ref{lemmaa} $(iv)$. Applying Ito's formula to the process $|Y_t ^n-P_T|^2$ we deduce, for $t\in [0,T]$, that
\begin{eqnarray}\label{jaj}
&&|Y_t ^n-P_T|^2 + \int _t ^T 
\|Z_s^n\|^2 ds + \int _t ^T \int _U |U_s^n(e)|^2 p(de, ds) \notag \\
&=&|\xi-P_T|^2 +  2 \int _t ^T \langle Y_s ^n-P_T, f(s,Y_s ^n,Z_s ^n,U_s^n)\rangle ds  - 2 \int _t ^T \langle Y_s ^n-P_T, Z_s^n dW_s \rangle\notag \\
 &&- 2 \int _t ^T \int _U  \langle Y_s ^n-P_T, U^n_s(e)  \rangle \mu(de,ds) - 2 n \int _t ^T  \langle Y_s ^n-P_T,Y_s^n-\pi(s,Y_s ^n)\rangle ds.
\end{eqnarray}
 Let
\begin{eqnarray*}
A_n:=|Y_t ^n-P_T|^2 + \int _t ^T \|Z_s^n\|^2 ds + \int _t ^T \int _U |U_s^n(e)|^2 p(de, ds).
\end{eqnarray*}
Rearranging \eqref{jaj}, we find that
\begin{eqnarray*}
&&A_n + 2 n \int _t ^T  \langle Y_s ^n-P_T,Y_s^n-\pi(s,Y_s ^n)\rangle ds \notag\\
&=& |\xi-P_T|^2 +  2 \int _t ^T \langle Y_s ^n-P_T, f(s,Y_s ^n,Z_s ^n,U_s^n)\rangle ds\notag\\
 &&- 2 \int _t ^T \langle Y_s ^n-P_T, Z_s^n dW_s\rangle - 2 \int _t ^T \int _U \langle Y_s^n-P_T, U^n_s(e)  \rangle \mu(de,ds).
\end{eqnarray*}
Furthermore, using Lemma \ref{lemmaa} $(iv)$ we see that
\begin{eqnarray}\label{jaj++ny}
&&A_n + 2 \gamma^{-1} n \int _t ^T |Y_s^n-\pi(s,Y_s ^n)|ds \notag\\
&\leq& |\xi-P_T|^2 +  2 \int _t ^T \langle Y_s ^n-P_T, f(s,Y_s ^n,Z_s ^n,U_s^n)\rangle ds\notag\\
 &&- 2 \int _t ^T \langle Y_s ^n-P_T, Z_s^n dW_s\rangle+ 2 \int _t ^T \int _U \langle Y_s^n-P_T, U^n_s(e)  \rangle \mu(de,ds).
\end{eqnarray}
Next, taking expectations in \eqref{jaj++ny} and using the fact that $\mu$ is a martingale measure, we can conclude that
\begin{eqnarray}\label{eq:inequal1}
&&E  \left [A_n+2 \gamma^{-1} n \int _t ^T |Y_s^n-\pi(s,Y_s ^n)|ds \right ]   \leq I_{t,T}
\end{eqnarray}
where \begin{eqnarray}\label{jaj++1}
I_{t,T}=E \left [ |\xi-P_T|^2 \right] + 2 E\left [\int _t ^T \langle Y_s ^n-P_T, f(s,Y_s ^n,Z_s ^n, U_s^n)\rangle ds\right].
\end{eqnarray}
Using the Lipschitz character of $f$, \eqref{data++} $(iii)$, and the inequality $ab \leq \eta a^2 + \frac{b^2}{4\eta}$ it follows that  we can estimate $I_{t,T}$ as,
\begin{eqnarray}\label{jaj++1+}
I_{t,T}&\leq& c\left( 1+ E \left[\int _t^T \left( |f(s,P_T,0,0)|^2+(1+2\eta)|Y^n_s-P_T|^2 ds \right ) \right] \right)\notag \\
&+& c E\left[ \int _t ^T \frac{1}{\eta} \left (\|Z_s^n\|^2 + \int _U |U_s^n(e)|^2 \lambda(de) \right) ds\right ]
\end{eqnarray}
where $\eta>0$ is a degree of freedom and $c$, $1\leq c<\infty$, is a constant depending on $\xi$, $f$ and $\mbox{diam($D_T$)}$. If we let $\eta$ be such that $c/\eta \leq1/2$, it follows from \eqref{eq:inequal1} and \eqref{jaj++1+}, after recalling the definition of $A_n$, that
\begin{eqnarray}\label{jaj++ha}
&&E[|Y_t ^n-P_T|^2]+ \frac 1 2 E \left [\int _t ^T \|Z_s^n\|^2 ds\right ] + \frac{1}{2}E \left [\int _t ^T \int _U |U_s^n(e)|^2 \lambda(de) ds \right ] \notag \\
&\leq& c \left( 1+ E \left [\int _t^T |Y^n_s-P_T|^2 ds \right ]\right).
\end{eqnarray}
Using \eqref{jaj++ha} and Gronwall's lemma we deduce that
\begin{equation}\label{ffa2--}
\sup_{0 \leq t \leq T}E[|Y_t^n-P_T|^2] \leq \tilde c\exp(\tilde c T)
\end{equation}
where $\tilde c$ is independent of $n$. In particular, we can conclude that there exists $\hat c$, $1\leq\hat c<\infty$, independent of $n$, such that
\begin{eqnarray}\label{ffa2--+}
(i)&&\sup_{0 \leq t \leq T}E[|Y_t^n|^2] \leq \hat c, \quad \mbox {and}\notag\\
(ii)&&\ E \left [\int _0 ^T \|Z_s^n\|^2 ds + \int _0 ^T\int _U |U_s^n(e)|^2 \lambda(de) ds \right]\leq\hat c.
\end{eqnarray}
We next prove that there exists $\check c$, $1\leq \check c<\infty$, independent of $n$, such that
\begin{eqnarray}\label{ffa2--++}
E[\sup_{0 \leq t \leq T}|Y_t^n|^2] \leq \check c.
\end{eqnarray}
To do this we first note, using \eqref{jaj++ny} and the Lipschitz character of
$f$, that, for a constant $c$, $1\leq c<\infty$,
\begin{eqnarray}\label{eq:BDGdoob}
&&c^{-1}|Y_t^n-P_T|^2\leq |\xi-P_T|^2+\biggl |\int _t ^T \langle Y_s ^n-P_T, f(s,P_T,0,0)\rangle ds\biggr |\notag\\
&&+\biggl |\int _t ^T |Y_s ^n-P_T|(|Y_s ^n-P_T|+\|Z_s ^n\|+\|U_s^n\|_2)ds\biggr |\notag\\
&&+\biggl |\int _t ^T \langle Y_s ^n-P_T, Z_s^n dW_s\rangle\biggr |+ \biggl | \int _t ^T \int _U \langle Y_s^n-P_T, U^n_s(e) \rangle \mu(de,ds)\biggr |.
\end{eqnarray}
We treat the last two terms on the right hand side of \eqref{eq:BDGdoob} using H\"{o}lders inequality and the Burkholder-Davis-Gundy inequality. Indeed, applying these yields 
\begin{eqnarray}\label{jaj++nyhha}
&&E\left [\sup_{0 \leq t \leq T}\left |\int _t ^T \langle Y_s ^n-P_T, Z_s^n dW_s\rangle\right |\right]\notag\\
&&\leq \biggl (E \left [\sup_{0 \leq t \leq T}|Y_t ^n-P_T|^2\right ]\biggr )^{1/2}\biggl (E\biggl[\sup_{0 \leq t \leq T}\biggl |\int _t ^T Z_s^n dW_s \biggr |^2
\biggr ]\biggr )^{1/2}\notag\\
&&\leq c\biggl (E \left [\sup_{0 \leq t \leq T}|Y_t ^n-P_T|^2 \right ]\biggr )^{1/2}\biggl (E \left [\int _0 ^T \|Z_s^n\|^2 ds\right ]\biggr )^{1/2}.
\end{eqnarray}
Similarly,
\begin{eqnarray*}
&&E\left [\sup_{0 \leq t \leq T}\biggl |\int _t ^T \int _U \langle Y_s^n-P_T, U^n_s(e)  \rangle \mu(de,ds)\biggr |\right]\notag\\
&&\leq c\biggl (E \left [\sup_{0 \leq t \leq T}|Y_t ^n-P_T|^2 \right ]\biggr )^{1/2}\biggl (E \left [\int _0 ^T\int _U |U_s^n(e)|^2 \lambda(de) ds\right ]\biggr )^{1/2}.
\end{eqnarray*}
Using the above estimates as well as  \eqref{ffa2--}, \eqref{ffa2--+} and assumption \eqref{timedep+} we get after taking expectation in \eqref{eq:BDGdoob} that 
\begin{eqnarray*}
&&E  \left [ \sup_{0\leq t \leq T}|Y_t^n-P_T|^2 \right ] \leq c \biggl (1+ \left ( E [\sup_{0 \leq t \leq T}|Y_t ^n-P_T|^2  ] \right )^{1/2} \biggr)
\end{eqnarray*}
from which we conclude that $E \left [ \sup_{0\leq t \leq T}|Y_t^n-P_T|^2 \right] \leq c$ for some constant $1 \leq c < \infty$ and, consequently, that \eqref{ffa2--++} holds.
Finally, starting from \eqref{jaj++ny} and repeating the arguments above we also deduce that
\begin{eqnarray*}
E \left [n\int _t ^T|Y_s^n-\pi(s,Y_s ^n)|ds \right ]\leq\hat c
\end{eqnarray*}
for some constant $\hat c$ independent of $n$. This completes the proof of Lemma \ref{ll1}.\end{proof}

\subsection{Uniform control of $d(Y_t^n,D_t)$}
We here prove the following lemma.

\begin{lemma}\label{lemmbb} Let $D\subset\mathbb{R}^{d+1}$ be a time-dependent domain satisfying \eqref{timedep+}~-~\eqref{timedep+2}. Let, for $\epsilon>0$ small,  $D_{\epsilon}$
be as in Lemma \ref{lemmauu1-}. Then there exist $\epsilon_0>0$ and $c$, $1\leq c<\infty$, both independent of $n$, such that
\begin{eqnarray*}
(i)&&E \left [\sup{_{0\leq t\leq T}d(Y_t^n,D_{\e,t})^2} \right ] \leq c\biggl ( \frac 1 n +\epsilon+n\epsilon^2\biggr ),\notag\\
(ii)&&E \left [\int_0^T(d(Y_t^n,D_{\e,t}))^2dt \right ] \leq c\biggl ( \frac 1 {n^2} +\frac \epsilon n+\epsilon^2\biggr ),
\end{eqnarray*}
whenever $0<\epsilon<\epsilon_0$ and $n\geq 1$.
\end{lemma}
\begin{proof} Let $\varphi_\epsilon(t,Y_t)=d(Y_t^n, D_{\e,t})^2=|Y_t^n-\pi_\epsilon(t,Y_t^n)|^2$. Then $\nabla_y\varphi_\epsilon(t,Y_t)=2(Y_t^n-\pi_\epsilon(t,Y_t^n))$. Using the
 Ito formula of Lemma \ref {lemmauu2+} we see that
\begin{eqnarray}\label{aaa-li}
&&\varphi_\e(t,Y_t^n) + \int_t^T (\partial_{s}\varphi_\epsilon)(s,Y_{s}^n)ds+ \frac{1}{2} \int_t^T  \sum_{i,j} (Z_s^{n} Z_s^{n,\ast})_{ij} \partial^2_{y_iy_j}\varphi_\epsilon(s, Y_{s}) ds \notag\\
&&+ \int _t ^T \int _U [\varphi_{\e}(s,Y_{s^-}^n+U^n_s(e)) - \varphi_{\e}(Y^n_{s^-}) - \langle \nabla \varphi_{\e}(s,Y^n_{s^-}) , U^n_s(e) \rangle ] p(de,ds) \notag \\
&&=\varphi_\epsilon(T,\xi)+I_1+I_2+I_3 + I_4,
\end{eqnarray}
where
\begin{eqnarray}\label{aaa+}
I_1&=&2\int _t ^T \langle Y_s^n-\pi_\epsilon(s,Y_s^n),f(s,Y_s,Z_s^n, U_s^n)\rangle ds,\notag\\
I_2&=& - 2\int _t ^T \langle Y_s^n-\pi_\epsilon(s,Y_s^n),Z^n_s
dW_s \rangle,\notag\\
I_3&=&- 2n \int_t ^T\langle Y_s^n-\pi_\epsilon(s,Y_s^n),Y_s^n-\pi(s,Y_s^n)\rangle ds, \notag \\
I_4&=&- 2 \int_t ^T \int _U \langle  Y_s^n-\pi_\epsilon(s,Y_s^n), U^n_s(e) \rangle \mu(de,ds).
\end{eqnarray}
Using Lemma \ref{lemmauu2ny} we see that
\begin{eqnarray*}
&&\int_t^T \partial_{s}\varphi_\epsilon(s,Y_s^n)ds+ \frac{1}{2} \int_t^T \sum_{i,j} (Z_s^{n} Z_s^{n,\ast})_{ij} \partial^2_{y_iy_j}\varphi_\epsilon(s, Y_s) ds \notag\\
&&+ \int _t ^T \int _U [\varphi_{\e}(s,Y_{s^-}^n+U^n_s(e)) - \varphi_{\e}(Y^n_{s^-} ) - \langle \nabla \varphi_{\e}(s,Y^n_{s^-}) , U^n_s(e) \rangle ] p(de,ds)\geq 0,
\end{eqnarray*}
and hence
\begin{eqnarray}\label{aaa}
\varphi_\e(t,Y_t^n)\leq \varphi_\epsilon(T,\xi)+I_1+I_2+I_3 + I_4.
\end{eqnarray}
Since $\xi \in D_T$ a.s. we see, using Lemma \ref{lemmauu1}, that  $\varphi _\e(T,\xi)\leq c \e^2$ a.s. To simplify the notation in what follows, we define $\chi_\epsilon (t,y) : [0,T]\times\mathbb R^d \to \{0,1\}$ as
\begin{equation*} 
\chi_\epsilon(t,y) =\begin{cases}1& \mbox{ if } d(y,D_{\epsilon,t})>\epsilon \\
 							0 & \mbox{otherwise}
\end{cases}.
\end{equation*}
We first focus on the term $I_1$ in \eqref{aaa+}. Then, using the above introduced notation we see that
\begin{eqnarray*}
I_1&\leq &2\int _t ^T |\varphi_\epsilon(s,Y_s^n)|^{1/2}|f(s,Y_s^n,Z_s^n, U^n_s)|\chi_\epsilon(s,Y_s^n)ds\notag\\
&&+2\int _t ^T |\varphi_\epsilon(s,Y_s^n)|^{1/2}|f(s,Y_s^n,Z_s^n, U_s^n)|(1-\chi_\epsilon(s,Y_s^n))ds.
\end{eqnarray*}
Furthermore, by the inequality $ab \leq \eta a^2 + \frac{b^2}{4\eta}$ and $x \leq \max\{1,x^2\}$,
\begin{eqnarray*}
(i)&& 2|\varphi_\epsilon(s,Y_s^n)|^{1/2}|f(s,Y_s^n,Z_s^n, U^n_s)|\chi_\epsilon(s,Y_s^n)\notag\\
&\leq&
\frac{n}{4} \varphi_\epsilon(s,Y_s^n) \chi_\epsilon(s,Y_s^n) + \frac{4}{n} |f(s,Y_s^n,Z_s^n, U^n_s)|^2 \chi_\epsilon(s,Y_s^n),\notag\\
(ii)&& 2|\varphi_\epsilon(s,Y_s^n)|^{1/2}|f(s,Y_s^n,Z_s^n, U^n_s)|(1-\chi_\epsilon(s,Y_s^n))\notag\\
&\leq&2({\e}+ {\e} |f(s,Y_s^n,Z_s^n,U^n_s)|^2)(1-\chi_\epsilon(s,Y_s^n)).
\end{eqnarray*}

Next, focusing on the term $I_3$, we have by the bilinearity of $\langle \cdot, \cdot\rangle$ that
 \begin{eqnarray}\label{aaa++}
I_3&=&-2n \int _t ^T|Y_s^n-\pi_\epsilon(s,Y_s^n)|^2\chi_\epsilon(s,Y_s^n) ds\notag\\
&&- 2n \int _t ^T\langle Y_s^n-\pi_\epsilon(s,Y_s^n),\pi_\epsilon(s,Y_s^n)-\pi(s,Y_s^n)\rangle\chi_\epsilon(s,Y_s^n) ds\notag\\
&&- 2n \int _t ^T\langle Y_s^n-\pi_\epsilon(s,Y_s^n),Y_s^n-\pi(s,Y_s^n)\rangle(1-\chi_\epsilon(s,Y_s^n)) ds\notag\\
&:=&I_{31}+I_{32}+I_{33}.
\end{eqnarray}
By Lemma \ref{lemmauu1} $(i)$ we immediately see that $|I_{33}|\leq c n\epsilon^2$. Furthermore, using Lemma \ref{lemmauu2} $(ii)$  we see that
\begin{eqnarray}\label{aaa+++}
|I_{32}|\leq c\sqrt{\epsilon} n \int _t ^T(d(Y_s^n,D_{\epsilon,s}))^{3/2}\chi_\epsilon(s,Y_s^n) ds.
\end{eqnarray}
Using the inequality \( ab \leq \frac{3a^{\frac{4}{3}}}{4}+\frac{b^4}{4} \) with
 $a=  d(Y^n_s, D_{\e,s})^{\frac{3}{2}}$ and $b=c\sqrt{\epsilon}$ we deduce from
 \eqref{aaa+++} that
    \begin{eqnarray}\label{aaa+++jj}
|I_{32}|\leq cn{\epsilon^2}+n \int _t ^T(d(Y_s^n,D_{\epsilon,s}))^{2}\chi_\epsilon(s,Y_s^n) ds.
\end{eqnarray}
Putting the estimate \eqref{aaa+++jj} into \eqref{aaa++} together we can conclude that
\begin{eqnarray*}
I_3\leq cn\e^2-n  \int _t ^T|Y_s^n-\pi_\epsilon(s,Y_s^n)|^2\chi_\epsilon(s,Y_s^n) ds.
\end{eqnarray*}
Hence, putting the estimate for $I_1$ and $I_3$ together we can conclude that
\begin{eqnarray}\label{aaad}
I_1+I_3&\leq & c\epsilon+cn\epsilon^2-\frac 3 4n \int _t ^T\varphi_\epsilon(s,Y_s^n)\chi_\epsilon(s,Y_s^n) ds\notag\\
&&+c\int _t ^T\biggl (\frac 1 n+\epsilon\biggr ) |f(s,Y_s^n,Z_s^n, U^n_s)|^2 ds.
\end{eqnarray}
Combining \eqref{aaa} and \eqref{aaad} we have proved that
\begin{eqnarray} \label{aaad+}
\varphi_\e(t,Y_t^n)&\leq& c( \e+n\e^2) -\frac{3}{4} n \int _t ^T \varphi_\e(s,Y^n_s) \chi_\e(s,Y^n_s) ds\notag\\
 &&+ c\int _t^T(\frac{1}{n} + \e) |f(s,Y^n_s, Z^n_s, U^n_s)|^2 ds \notag \\
&&- 2\int _t ^T \langle Y_s^n-\pi_\epsilon(s,Y_s^n),Z^n_s
dW_s \rangle\notag\\
&&-2\int_t ^T \int _U \langle Y_s^n-\pi_\epsilon(s,Y_s^n), U^n_s(e) \rangle  \mu(de,ds).
\end{eqnarray}
In particular,
\begin{eqnarray} \label{aaad+ia}
&&\varphi_\e(t,Y_t^n)+\frac{3}{4}n \int _t ^T \varphi_\e(s,Y^n_s) \chi_\e(s,Y^n_s) ds\notag\\
 &\leq& c( \e+n\e^2)+ c\int _t^T(\frac{1}{n} + \e) |f(s,Y^n_s, Z^n_s, U^n_s)|^2 ds \notag \\
&&-2\int _t ^T \langle Y_s^n-\pi_\epsilon(s,Y_s^n),Z^n_s
dW_s \rangle-2\int_t ^T \int _U \langle Y_s^n-\pi_\epsilon(s,Y_s^n), U^n_s(e) \rangle  \mu(de,ds).
\end{eqnarray}
where $c$ is independent of $\e$ and $n$. The estimate in Lemma \ref{lemmbb} $(ii)$ now follows from taking expectation in \eqref{aaad+ia} and using the Lipschitz property of $f$ and
Lemma \ref{ll1} $(i)$ and $(ii)$. Similarly, using  \eqref{aaad+} we can conclude that
\begin{eqnarray*}
\sup_{0\leq t\leq T}E[\varphi_\e(t,Y^n_t)] \leq c\bigl (\frac{1}{n} +\e + n\e^2\bigr),
\end{eqnarray*}
for a constant $c$, independent of  $\e$ and $n$. Since $\varphi_\e(t, Y_t) \geq 0$ for all $t \in [0,T]$ we can, repeating the arguments above, also conclude from \eqref{aaa-li} that 
\begin{eqnarray}\label{eq:allests}
&& E\left [ \int_t^T (\partial_{s}\varphi_\epsilon)(s,Y_{s}^n)ds \right ]+ E \left [\int_t^T  \sum_{i,j} (Z_s^{n} Z_s^{n,\ast})_{ij} \partial^2_{y_iy_j}\varphi_\epsilon(s, Y_{s}) ds \right ]  \notag\\
+&&E\left [ \int _t ^T \int _U [\varphi_{\e}(s,Y_{s^-}^n+U^n_s(e)) - \varphi_{\e}(Y^n_{s^-}) - \langle \nabla \varphi_{\e}(s,Y^n_{s^-}) , U^n_s(e) \rangle ] p(de,ds) \right ] \notag  \\
&& \leq c\bigl (\frac{1}{n} +\e + n\e^2\bigr),
\end{eqnarray}
for some constant $1\leq c < \infty$, for all $t \in [0,T]$. 
Once again using the Lipschitz property of $f$ and Lemma \ref{ll1} in \eqref{aaad+ia} we see that to complete the proof of Lemma \ref{lemmbb} $(i)$ it remains to control the terms
\begin{eqnarray}\label{eq:finests}
&(i)&E \left [\sup_{0\leq t \leq T} \left | \int _t ^T \langle Y_s^n-\pi_\epsilon(s,Y_s^n),Z^n_s dW_s \rangle  \right | \right ] 
\notag \\
&(ii)&E \left [ \sup_{0\leq t \leq T}\left |\int_t ^T \int _U \langle Y_s^n-\pi_\epsilon(s,Y_s^n), U^n_s(e) \rangle \mu(de,ds) \right | \right ].
\end{eqnarray}

We first treat \eqref{eq:finests} $(i)$. As in \eqref{jaj++nyhha} we use the Burkholder-Davis-Gundy inequality to see that
\begin{align}\label{eq:nuzest}
&E \left [\sup_{0\leq t \leq T} \left | \int _t ^T \langle Y_s^n-\pi_\epsilon(s,Y_s^n),Z^n_s dW_s \rangle \right | \right ] &\notag \\
  \leq & E \left [ \left ( \int _0 ^T \left |(Y_s^n-\pi_\epsilon(s,Y_s^n))^\ast Z^n_s \right|^2 ds \right )^{1/2} \right ].
\end{align}
By the convexity of $\varphi_\e$ and equivalence of Euclidean norms (see \cite{GP}~p.~115) we have that
$$
\frac {|(Y_s^n-\pi_\e(s,Y_s^n))^\ast Z^n_s|^2}{|(Y_s^n-\pi_\e(s,Y_s^n))|^2} \I_{\{Y^n_s \not \in D_{e,s}\}}\leq c \left ( \sum_{i,j} (Z_s^{n} Z_s^{n,\ast})_{ij} \partial^2_{y_iy_j}\varphi_\epsilon(s, Y_{s}) \right),
$$
where $\I$ is the indicator function, i.e., $\I_{\{Y^n_s \not \in D_{\e,s}\}} = 1$ if $Y^n_s \not \in D_{\e,s}$ and $0$ otherwise. Hence, it follows from \eqref{eq:allests} that for $t\in [0,T]$
\begin{equation}\label{eq:nuzest1+}
\int_t ^T \frac{|(Y_s^n-\pi_\e(s,Y_s^n))^\ast Z^n_s|^2}{|(Y_s^n-\pi_\e(s,Y_s^n))|^2}  \I_{\{Y^n_s \not \in D_{e,s}\}}ds\leq  c\bigl (\frac{1}{n} +\e + n\e^2\bigr).
\end{equation}
Note also that $\left | Y_s^n-\pi_\epsilon(s,Y_s^n) \right |= \left |Y_s^n-\pi_\epsilon(s,Y_s^n)  \right | \I_{\{Y^n_s \not \in D_{\e,s}\}}$ and that
\begin{align}\label{eq:nuzest2}  \int _t ^T |(Y_s^n-\pi_\e(s,Y_s^n))^\ast Z^n_s|^2 ds & \notag \\
 \leq \sup_{t\leq s \leq T}& \varphi_\e(s, Y^n_s)  \int _t ^T \frac {| (Y_s^n-\pi_\e(s,Y_s^n))^\ast Z^n_s|^2}{|(Y_s^n-\pi_\e(s,Y_s^n))|^2} \I_{\{Y^n_s \not \in D_{e,s}\}} ds.
\end{align}
We again use $ab \leq \eta a^2 + \frac{b^2}{\eta}$ to conclude from \eqref{eq:nuzest}, \eqref{eq:nuzest1+} and \eqref{eq:nuzest2} that 
\begin{align}\label{eq:wrty}
&E \left [\sup_{0\leq t \leq T} \left | \int _t ^T \langle Y_s^n-\pi_\epsilon(s,Y_s^n),Z^n_s dW_s \rangle \right | \right ] \notag \\
\leq &\eta E \left [\sup_{0 \leq t \leq T} \varphi_\e(t, Y^n_t) \right ] +c_\eta(\frac{1}{n} +\e + n\e^2\bigr)
\end{align}
for some constant $c_\eta <\infty$ depending on the degree of freedom $\eta >0$.

We now treat the term \eqref{eq:finests} $(ii)$. Using Taylor's theorem we see that 
\begin{align*}E\left [ \int _t ^T \int _U [\varphi_{\e}(s,Y_{s^-}^n+U^n_s(e)) - \varphi_{\e}(Y^n_{s^-}) - \langle \nabla \varphi_{\e}(s,Y^n_{s^-}) , U^n_s(e) \rangle p(de,ds) \right ]& \notag \\
\geq \breve c E \left [ \int_t^T |U(e)|^2  \lambda(de) ds \right ],
\end{align*}
for some constant $\breve c\geq 0$. Furthermore, by the strong convexity of $\varphi_\e$, and this is a consequence of \eqref{timedep+}, there exists a constant $\kappa >0$ such that $\breve c\geq \kappa >0$. Therefore we can, in a way similar to the above, conclude that
\begin{align}\label{eq:wrty2}
&E \left [ \sup_{0\leq t \leq T}\left |\int_t ^T \int _U \langle Y_s^n-\pi_\epsilon(s,Y_s^n), U^n_s(e) \rangle \mu(de,ds) \right | \right ] \notag \\
\leq & E \left [\left (\int_t ^T \int _U |Y_s^n-\pi_\epsilon(s,Y_s^n)|^2 |U^n_s(e)|^2 \lambda(de) ds  \right) ^{1/2} \right ] \notag  \\
\leq & E \left [ \left ( \sup_{t\leq s \leq T}\varphi_\e(s,Y^n_s) \int_t ^T \int _U  |U^n_s(e)|^2 \lambda(de) ds  \right) ^{1/2}\right ] \notag  \\
\leq &\eta E \left [\sup_{0 \leq t \leq T} \varphi_\e(t, Y^n_t) \right ] +c_\eta(\frac{1}{n} +\e + n\e^2\bigr).
\end{align}
Once again, $c_\eta <\infty$  is a constant depending on the degree of freedom $\eta>0$. 

The proof of Lemma \ref{lemmbb} is now completed by choosing $\eta$ small enough (in analogy with \eqref{jaj++1+},\eqref{jaj++ha}) and combining \eqref{aaad+ia} with \eqref{eq:wrty}, \eqref{eq:wrty2}.
\end{proof}
\begin{lemma}\label{lemmbbaa} Let $D\subset \R^{d+1}$ be a time-dependent domain satisfying \eqref{timedep+}~-~\eqref{timedep+2}. Then there exists  $c$, $1\leq c<\infty$, independent of $n$ such that
\begin{eqnarray*}
(i)&&E \left [\sup_{0\leq t\leq T}(d(Y_t^n,D_{t}))^2 \right] \leq \frac{ c} n,\notag\\
(ii)&&E \left [\int_0^T(d(Y_t^n,D_{t}))^2dt \right] \leq \frac {c} {n^2},
\end{eqnarray*}
whenever $n\geq 1$.
\end{lemma}
\begin{proof}  Let, for $\epsilon>0$ small,  $D_{\epsilon}$
be as in Lemma \ref{lemmauu1-}. Then, using Lemma \ref{lemmauu1-} we have
\begin{eqnarray*}
h(D_{t},D_{\epsilon,t})<\epsilon\mbox{ for all $t\in[0,T]$.}
\end{eqnarray*}
Hence,
\begin{eqnarray*}
d(Y_t^n,D_{t})\leq d(Y_t^n,D_{\epsilon,t})+\epsilon\mbox{ for all $t\in[0,T]$.}
\end{eqnarray*}
Applying Lemma \ref{lemmbb} and letting $\e \to 0$ completes the proof.
\end{proof}

\subsection{$(Y_t^n, Z_t^n, U_t^n)$ is  a Cauchy sequence }

\begin{lemma}\label{anotherlemma} There exists a constant $c$ such that the following holds whenever $m, n \in \mathbb{Z}^+$:
\begin{eqnarray*}
E \left [\sup_{0 \leq t \leq T} |Y_t^n-Y_t^m|^2 + \int_0 ^t \|Z_t^n-Z_t^m\|^2  dt\right ] &\leq& c \left(\frac{1}{n} + \frac{1}{m}\right),\\
E\left [\int_0 ^T \int _U |U_s^n(e)-U^m_s(e)| ^2 \lambda(de) ds\right] &\leq& c \left(\frac{1}{n} + \frac{1}{m}\right).
\end{eqnarray*}
\end{lemma}
\begin{proof}Applying Ito's formula to $|Y^n_t-Y^m_t|^2$ we get that
\begin{eqnarray}\label{eq:itoexpmn}
&&|Y^n_t-Y^m_t|^2+\int _t ^T \|Z_s^n-Z_s^m\|^2 ds\notag\\
&& + \int _t ^T \int _U |U^n_s(e)- U^m_s(e)| ^2 p(de,ds) \notag\\
&=& 2 \int _t ^T \langle Y^n_s - Y^m_s , f(s,Y^n_s, Z^n_s, U^n_s) - f(s,Y^m_s, Z^m_s, U^m_s) \rangle ds \notag \\
&&-  2 \int _t ^T \langle Y^n_s - Y^m_s , (Z^n_s- Z^m_s) dW_s\rangle\notag\\
 &&- 2 \int _t ^T \int _U \langle Y^n_s - Y^m_s , U^n_s(e)- U^m_s(e) \rangle \mu(de,ds)  \notag \\
&&-  2n \int _t ^T \langle Y^n_s - Y^m_s , Y^n_s- \pi(s,Y^n_s) \rangle ds\notag\\
&& + 2m \int _t ^T \langle Y^n_s - Y^m_s , Y^m_s- \pi(s,Y^m_s) \rangle ds.
\end{eqnarray}
Hence, taking expectation and using the Lipschitz character of $f$ we deduce that
\begin{eqnarray}\label{apa1}
&&E \left [|Y^n_t-Y^m_t|^2 \right]+E\left [\int _t ^T \|Z_s^n-Z_s^m\|^2 ds\right ]\notag\\
&& + E\left [\int _t ^T \int _U |U^n_s(e)- U^m_s(e)| ^2 \lambda(de)ds\right] \notag\\
&\leq& c E\left [ \int _t ^T (|Y^n_s - Y^m _s |^2 + |Y^n_s - Y^m _s |\|Z^n_s - Z^m _s \|)ds\right ]\notag \\
&&+c E\left [ \int _t ^T |Y^n_s - Y^m _s | \int _U |U^n_s(e) - U^m _s(e)| \lambda(de) ds\right ] \notag \\
&&- 2E\left [ \int _t  ^T \langle Y^n_s - Y^m_s , n(Y^n_s - \pi(s,Y^n_s)) - m(Y^m_s - \pi(s,Y^m_s)) \rangle ds\right ].
\end{eqnarray}
Note that
\begin{eqnarray*}
&&-\langle Y^n_s - Y^m_s , n(Y^n_s - \pi(s,Y^n_s)) - m(Y^m_s - \pi(s,Y^m_s)) \rangle \notag\\
&=&\langle Y^m_s - Y^n_s , n(Y^n_s - \pi(s,Y^n_s)) \rangle +\langle Y^n_s - Y^m_s ,  m(Y^m_s - \pi(s,Y^m_s)) \rangle.
\end{eqnarray*}
Using Lemma \ref{lemmaa} $(ii)$ we have that
\begin{eqnarray*}
&&\langle Y^m_s - Y^n_s , n(Y^n_s - \pi(s,Y^n_s)) \rangle\leq n \langle Y^m_s - \pi(s,Y^m_s) , Y^n_s - \pi(s,Y^n_s)\rangle \notag\\
&&\langle Y^n_s - Y^m_s ,  m(Y^m_s - \pi(s,Y^m_s)) \rangle\leq m \langle Y^n_s - \pi(s,Y^n_s) , Y^m_s - \pi(s,Y^m_s)\rangle.
\end{eqnarray*}
Furthermore,
\begin{eqnarray*}
&&2E\left [ \int _t  ^T \langle Y^m_s - Y^n_s , n(Y^n_s - \pi(s,Y^n_s)) \rangle ds\right ]\notag\\
&\leq& 2nE\left [ \int _t  ^T |Y^m_s - \pi(s,Y^m_s)|| Y^n_s - \pi(s,Y^n_s)|ds\right ]\notag\\
&\leq & nE\left [ \int _t  ^T \beta(d(Y^m_s,D_s))^2+\beta^{-1}(d(Y^n_s,D_s))^2ds\right ]\notag\\
&\leq &c(n\beta m^{-2}+\beta^{-1} n^{-1})\leq cm^{-1}
\end{eqnarray*}
where we have used Lemma \ref{lemmbbaa} $(ii)$ and chosen the degree of freedom to equal $\beta=m/n$. This argument can be repeated with
$n\langle Y^m_s - Y^n_s , n(Y^n_s - \pi(s,Y^n_s)) \rangle$ replaced by $m \langle Y^n_s - \pi(s,Y^n_s) , Y^m_s - \pi(s,Y^m_s)\rangle $ resulting in the bound $ cn^{-1}$. Put together we can conclude that
\begin{eqnarray}\label{apa}
&&- 2E\left [ \int _t  ^T \langle Y^n_s - Y^m_s , n(Y^n_s - \pi(s,Y^n_s)) - m(Y^m_s - \pi(s,Y^m_s)) \rangle ds \right ]\notag\\
&\leq &c(n^{-1}+m^{-1}).
\end{eqnarray}
Combining \eqref{apa1}, \eqref{apa} and using Cauchy's inequality as in \eqref{jaj++1+}, \eqref{jaj++ha} we can conclude that
\begin{eqnarray}\label{apa1+}
&&E \left [|Y^n_t-Y^m_t|^2 \right ]+ \frac{1}{2}E\left [\int _t ^T\|Z_s^n-Z_s^m\|^2 ds\right ]\notag\\
&& +\frac{1}{2} E\left [\int _t ^T \int _U |U^n_s(e)- U^m_s(e)| ^2 \lambda(de)ds\right ] \notag\\
&\leq& c E\left [ \int _t ^T |Y^n_s - Y^m _s |^2 ds\right ]+c(n^{-1}+m^{-1})
\end{eqnarray}
where $c$ is independent of $n$ and $m$. By Gronwall's inequality we then have, using \eqref{apa1+},
\begin{eqnarray}\label{apa1++}
E \left [|Y^n_t-Y^m_t|^2 \right ]\leq c(n^{-1}+m^{-1}).
\end{eqnarray}
Subsequently,
\begin{eqnarray}\label{apa1+a}
E\left [\int _t ^T\|Z_s^n-Z_s^m\|^2 ds\right ]&\leq& c(n^{-1}+m^{-1})\notag\\
E\left[\int _t ^T \int _U |U^n_s(e)- U^m_s(e)| ^2 \lambda(de)ds\right ]&\leq& c(n^{-1}+m^{-1}).
\end{eqnarray}
Using \eqref{apa1++}, \eqref{apa1+a}, Lemma \ref{lemmbbaa} $(i)$, and now familiar arguments based on the Burkholder-Davis-Gundy inequality, we can, starting from \eqref{eq:itoexpmn}, also deduce that
 \begin{equation*}
E  \left [ \sup_{0 \leq t \leq T} |Y^n_t -Y^m_t|^2 \right ] \leq c \left (\frac{1}{n}+\frac{1}{m}\right),
\end{equation*}
to complete the proof of Lemma \ref{anotherlemma}. We omit further details.
\end{proof}

\section{The final argument: proof of Theorem \ref{maint1}} \label{sec:proof}
\noindent
Using Lemma \ref{anotherlemma} we can conclude that $(Y_t^n,Z_t^n,U_t^n)$ is a Cauchy sequence in the space of progressively measurable processes $(Y_t,Z_t,U_t)$ satisfying
\begin{eqnarray*}
E \left [\sup_{0 \leq t\leq T} |Y_t|^2 \right] +E \left [\int _0 ^T \|Z_t\|^2 ds \right] +E \left[\int _0 ^T \int _U |U_t(e)|^2 \lambda(de) ds \right]<\infty.
\end{eqnarray*}
Hence, taking a subsequence if necessary, we have a sequence $(Y_t^n,Z_t^n,U_t^n)_{n \geq 0}$ and a triple of processes $(Y_t,Z_t,U_t)$ such that
\begin{equation*}
Y_t=\lim _{n \to \infty} Y^n_t, \
Z_t=\lim _{n \to \infty} Z^n_t, \
U_t=\lim _{n \to \infty} U^n_t
\end{equation*}
in the sense that
\begin{eqnarray}\label{aadg}
(i)&&E \left [\sup_{0 \leq t\leq T} |Y_t^n-Y_t|^2 \right] \to 0,\notag\\
(ii)&&E\left [\int _0 ^T \|Z_t^n-Z_t\|^2 ds \right] \to 0, \notag \\
(iii)&& E \left [\int _0 ^T \int _U |U_t^n(e)-U_t(e)|^2 \lambda(de) ds \right] \to 0
\end{eqnarray}
as $n \to \infty$. Furthermore, by Lemma \ref{ll1} we have
\begin{eqnarray}\label{ffaagain}
(i)&&E \left[\sup_{0 \leq t \leq T} |Y_t|^2 \right] <\infty,\notag\\
(ii)&& E \left [\int_0 ^T \|Z_t\|^2 dt +  \int _0^T \int _U|U_s(e)|^2 \lambda(de)ds \right ] < \infty.
\end{eqnarray}
Note that, as a uniform limit of c\`{a}dl\`{a}g functions $\{Y^n_t\}$, we immediately have that $Y_t\in \mathcal{D}\left( \left[ 0,T\right] ,\R^{d}\right)$ and, by Lemma \ref{lemmbbaa} $(i)$, we have $Y_t \in \overline D$. Recall that
\begin{equation*}
\Lambda_t^n= -n \int _0 ^t (Y^n_s- \pi(s,Y^n_s))ds=\int_0^t -\frac{( Y_s ^n-\pi(s,Y_s ^n))}{| Y_s ^n-\pi(s,Y_s ^n)|}d|\Lambda^n|_s
\end{equation*}
and that $-( Y_s ^n-\pi(s,Y_s ^n))/{| Y_s ^n-\pi(s,Y_s ^n)|}$ is an element in the inward directed normal cone to $D_s$ at
$\pi(s,Y_s ^n)\in\partial D_s$. Using \eqref{ffa} $(iii)$, \eqref{aadg} and \eqref{ffaagain} it follows that  there exists $\Lambda_t$ such that
\begin{equation*}
\Lambda_t=\lim _{n \to \infty}\Lambda^n= \lim _{n \to \infty} -n \int _0 ^t (Y^n_s- \pi(s,Y^n_s))ds
\end{equation*}
in the sense that
\begin{equation*}
E \left [\sup_{0 \leq t\leq T} |\Lambda_t^n-\Lambda_t|^2 \right ] \to 0.
\end{equation*}
Hence, as $(\Lambda^n_t(\omega))_{0\leq t \leq T}$ is continuous,  $(\Lambda_t(\omega))_{0\leq t \leq T}$  is continuous in $t$ a.s.

\subsection{Existence: $(Y_t, Z_t, U_t, \Lambda_t)$ is a solution} We will now prove that the constructed quadruple $(Y_t,Z_t,U_t,\Lambda_t)$ is a solution to our original problem. We first note that, as a limit of $(Y^n_t, Z^n_t, U^n_t)$, $(Y_t,Z_t,U_t)$ are progressively measurable, $Y_t\in \mathcal{D}\left( \left[ 0,T\right] ,\R^{d}\right)$ and $Z$ and $U$ are predictable. Hence it remains  to verify that $(Y_t, Z_t, U_t, \Lambda_t)$ satisfies $(i)$-$(vi)$ stated in Definition \ref{rbsde} and that $\Lambda\in \mathcal{BV}\left( \left[ 0,T\right] ,%
\mathbb{R}
^{d}\right) $. That
 $(Y_t, Z_t, U_t, \Lambda_t)$ satisfies $(i)$-$(iii)$ was proved above and $(iv)$ is a consequence of Lemma \ref{lemmbbaa} and
 \eqref{aadg} $(i)$. Hence we in the following focus on properties $(v)$ and $(vi)$. As mentioned above, we have that $(\Lambda_t(\omega))_{0\leq t \leq T}$  is continuous in $t$ for almost all
 $\omega$ by uniform convergence. Furthermore, using
that
$$\int _t ^T d|\Lambda^n|_s=n \int _t ^T |Y_s^n - \pi(s, Y^n _s)|ds$$ we see from Lemma \ref{ll1} $(iii)$
that
\begin{equation*}
E \left[\int _0 ^T d|\Lambda^n|_s \right ] = E \left [ n \int _0 ^T |Y^n_s-\pi(s,Y^n_s)| ds \right ]\leq c\mbox{ for all $n\in\mathbb Z_+$},
\end{equation*} 
\noindent for some constant $c$ which is independent of $n$. It follows that  $\Lambda_t(\omega)$ is
of bounded total variation on $[0,T]$ for almost all
 $\omega$. Hence, it only remains to verify that
\begin{eqnarray}\label{defo}
(v)&&\Lambda_t=\int_{0}^{t}\gamma _{s}d\left\vert \Lambda \right\vert
_{s},\ \gamma _{s}\in N_{s}^{1}\left( Y_{s}\right) \ \mbox{ whenever } Y_s \in \partial D_{s},\notag\\
(vi)&& d\left\vert \Lambda \right\vert \left( \left\{ t\in \left[ 0,T\right]
:\left(t,Y_{t}\right) \in D\right\} \right)=0.
\end{eqnarray}
To verify the statements in \eqref{defo} we will use the following lemma.
\begin{lemma} \label{lemma:adgda}
Let $\{\Lambda^n\}_{n \in \mathbb{Z}_+}$ be a sequence of continuous functions, $\Lambda^n:[0,T]\to \R^d$, which converges
uniformly to $\Lambda$ as $n\to \infty$. Assume  $\Lambda^n\in \mathcal{BV}\left( \left[ 0,T\right] ,%
\mathbb{R}
^{d}\right) $ and that $|\Lambda^n|_T \leq c$, for some $c<\infty$, hold for all $n$. Let
 $\{f^n\}_{n \in \mathbb{Z}_+}$ be a sequence of c\`{a}dl\`{a}g functions, $f^n:[0,T]\to \R^d$, converging uniformly to $f$ as $n\to \infty$. Then,
\begin{equation*}
\lim _{n \to \infty}\int _0 ^t \langle f^n_s, d\Lambda^n_s \rangle  = \int _0 ^t \langle f_s, d\Lambda_s \rangle
\end{equation*}
for all $t \in [0, T]$.
\end{lemma}
\begin{proof} This is essentially Lemma 5.8 in \cite{GP}, see also \cite{S}.
\end{proof}

 Using Lemma \ref{lemmaa} $(i)$ we see that ,
 $$\langle z_t- Y^n _t, Y^n_t-\pi(t,Y^n_t) \rangle \leq 0$$
for any c\`{a}dl\`{a}g process $z_t$  taking values in $\overline{D_t}$. Hence, for any such process $z_t$ we have that 
  \begin{eqnarray*}
   \int_0 ^t  -n \langle z_s- Y^n _s, Y^n_s-\pi(s,Y^n_s) \rangle ds =\int_0^t \langle z_s-Y^n_s , d\Lambda^n_s \rangle \geq 0.
   \end{eqnarray*}
Passing to the limit we obtain, using Lemma \ref{lemma:adgda}, that
\begin{equation} \label{defo+}
\int_0 ^ t\langle Y_s- z_s, d\Lambda_s \rangle \leq 0
\end{equation}
for all $z \in \mathcal{D}\left( \left[ 0,T\right], \R^{d}\right)$ taking values in $\overline D$, and for all $t\in [0,T]$.
\noindent Next, let $\tau \in [0,T)$ be any time such that $Y_{\tau} \in D$ and let $\hat \gamma$ be a unit vector in $\R^d$. Since $Y_s$ is right-continuous, taking assumption \eqref{limitzero} into account, we see that there exists $\e >0$ and $\delta > 0$ such that $Y_s+\e \hat \gamma \in D$ and $Y_s-\e \hat \gamma \in D$ whenever $s \in [\tau, \tau+\delta]$.  However, this in combination with \eqref{defo+} implies that
\begin{equation*}
0 \leq \int _{\tau} ^{\tau+\delta} \hat \gamma d\Lambda_t \leq 0,
\end{equation*}
which in turn implies $(vi)$ in \eqref{defo}. Hence
 \begin{eqnarray*}
\Lambda_t=\int_{0}^{t}\gamma _{s}d\left\vert \Lambda \right\vert_{s},
\end{eqnarray*}
for some vector field $\gamma_s\in \R^d$, with support on $\partial D$, and such that $|\gamma_{s}|=1$ (see \eqref{bff}). To conclude the existence part of Theorem \ref{maint1} it remains to show \eqref{defo}~$(v)$, i.e., that $\gamma_s\in N_{s}^{1}\left( Y_{s}\right)$ whenever $Y_s \in \partial D_{s}$. However, using the above and \eqref{bff} we see that to prove $\gamma_s\in N_{s}^{1}\left( Y_{s}\right)$ whenever $Y_s \in \partial D_{s}$, it is enough to prove that if  $ Y_{s}\in\partial D_s$ and if
$\langle Y_{s}-z_s,\gamma_s\rangle\leq 0$
for all $z_s \in D_s$, then ${\gamma_s}\in N_{s}\left( Y_{s}\right)$. To do this, take $\beta=Y_s-\gamma_s \in  \mathbb R^d$. Then, 
\begin{equation*}
|\beta-z_s|^2=|\beta-Y_{s}|^2+|Y_{s}-z_s|^2 + 2\langle \beta -Y_{s},Y_{s}-z_s\rangle
\end{equation*}
for all $z_s\in D_s$. Hence, if $\langle \beta -Y_{s},Y_{s}-z_s\rangle = -\langle Y_{s}-z_s,\gamma_s\rangle\geq 0$, then we have that
\begin{equation*}
|\beta-z_s|^2\geq|\beta-Y_{s}|^2
\end{equation*}
for all $z_s \in D_s$, which implies $\gamma_s \in N_s(Y_s)$. This proves \eqref{defo} $(v)$ and thus the proof of the existence part of Theorem \ref{maint1} is complete. \hfill $\Box$

\subsection{Uniqueness: $(Y_t, Z_t, U_t, \Lambda_t)$ is the only solution}
We here prove the uniqueness part of Theorem \ref{maint1} using Ito's formula. Indeed, assume that $(Y^i,Z^i, U^i , \Lambda^i)$, $i=1,2$, are two solutions to the
 the reflected BSDE under consideration and define
\begin{equation*}
\{\Delta Y_t, \Delta Z_t, \Delta U_t, \Delta \Lambda_t\}= \{Y^1_t - Y^2_t,  Z^1_t - Z^2_t, U^1_t - U^2_t, \Lambda^1_t - \Lambda^2_t \}.
\end{equation*}
Then, applying Ito's formula to $|\Delta Y_t|^2$ and taking expectation we have that
\begin{align*}
& E \left [ |\Delta Y_t|^2+\int _t ^T |\Delta Z_s|^2 ds + \int _t ^T  \int _ U |\Delta U_s(e)|^2 \lambda(de)ds \right ] \\
=& 2 E  \left [ \int _ t ^T \langle \Delta Y_s, f(s,Y_s^1, Z_s^1, U_s^1)-f(s, Y_s^2,Z_s^2,U_s^2) \rangle ds \right  ]\\
&+ 2 E \left [\int _t ^T \langle \Delta Y_s, d\Delta\Lambda_s \rangle  \right].
\end{align*}
Using \eqref{defo+} we see that the last term on the right hand side in the above display is $\leq0$. Next, using the Lipschitz character of $f$ and standard manipulations, see \eqref{jaj++1}, \eqref{jaj++1+}, we can conclude that
\begin{align*}
& E \left [ |\Delta Y_t|^2+\int _t ^T |\Delta Z_s|^2 ds + \int _t ^T  \int _ U |\Delta U_s(e)|^2 \lambda(de)ds \right ] \\
\leq& c E\left [ \int _ t ^T |\Delta Y_s|^2 ds + \frac{1}{2}  \int _t ^T |Z_s| ^2 ds +  \frac{1}{2} \int _t ^T |\Delta U_s(e)|^2 \lambda(de)ds \right ].
\end{align*}
Applying Gronwall's lemma we see that $\{\Delta Y_t, \Delta Z_t, \Delta U_t\}$ must be identically zero a.s. By $(iii)$ of Definition 1, the same applies to $\Delta \Lambda_t$ and the proof is hence complete. \hfill $\Box$


\begin{thebibliography}{99}

\bibitem[AF]{AF} B. El Asri and I. Fakhouri, {\em Optimal multi-modes switching with the switching cost not necessarily positive}, Arxiv preprint arXiv:1204.1683v1, 2012 - arxiv.org.

\bibitem[AH]{AH} B. El-Asri and S. Hamadene,
\emph{The finite horizon optimal multi-modes switching problem: The viscosity solution approach}, Applied Mathematics \& Optimization {\bf 60} (2009), 213-235.

\bibitem[BBP]{BBP} G. Barles, R. Bukhdan and E. Pardoux, {\em BSDE's and integral-partial differential equations}, Stochastics and Stochastics Report, {\bf 60} (1997), 57-83.

\bibitem[C]{C} C. Costantini, {\em The Skorohod oblique reflection problem in domains with
corners and application to stochastic differential equations},  Probability Theory
and Related Fields, {\bf 91} (1992), 43-70.

\bibitem[CGK]{CGK} C. Costantini, E. Gobet, and N. El Karoui. {\em Boundary sensitivities for diffusion processes in time dependent domains}, Applied Mathemathics
\& Optimization, {\bf 54} (2006), 159-187.

\bibitem[CK]{CK} J. Cvitani´c and I. Karatzas, {\em Backward stochastic differential equations with reflection and Dynkin games}, The Annals of Probability, {\bf 24} (1996), 2024-2056.

\bibitem[DHP]{DHP}
B. Djehiche, S. Hamadene and A. Popier, \emph{A finite horizon optimal multiple switching problem},
SIAM Journal on Control and Optimization, {\textbf 48} (2010), 2751-2770.

\bibitem[EKPPQ]{EKPPQ}
N. El-Karoui, C. Kampoudjian, E. Pardoux, S. Peng and M.C. Quenez
\emph{Reflected Solutions of backward SDE's and related obstacle problems for PDE's},
The Annals of Probability, {\bf 25} (1997) 702-737.

\bibitem[EPQ]{EPQ} N. El Karoui, S. Peng and M.C. Quenez. {\em Backward stochastic differential equations in finance}, Mathematical finance, {\bf 7} (1997), 1-71.


\bibitem[GP]{GP} A. Gegout-Petit and E. Pardoux, {\em Equations differentielles stochastiques retrogardes reflechies dans un convexe}, Stochastics and Stochastics Reports, {\bf 57} (1996), 111-128.

\bibitem[HL]{HL} S. Hamadene and J. Lepeltier, {\em Zero-zum stochastic differential games and backward equations}, Systems \& Control Letters, {\bf 24} (1995), 259-263.

\bibitem[HT]{HT} Y. Hu and S. Tang
\emph{Multi-dimensional BSDE with oblique reflection and optimal switching}, Probability Theory and Related Fields, {\bf147} (2010), 89-121.

\bibitem[HZ]{HZ} S. Hamadene and J. Zhang
\emph{Switching problem and related system of reflected backward SDEs}, Stochastic Processes and their Applications, {\bf120} (2010), 403-426.

\bibitem[LS]{LS} P. L. Lions and A. S. Sznitman, {\em Stochastic differential equations with reflecting boundary conditions}, Communications on Pure and Applied Mathematics, {\bf 37} (1984), 511-537.

\bibitem[NO]{NO} K. Nystr{\"o}m and T. {\"O}nskog, {\em The Skorohod Oblique Reflection Problem in Time-dependent Domains}, Annals of Probability
{\bf 38} (2010), 2170-2223.

\bibitem[O]{O} Y. Ouknine, {\em Reflected backward stochastic differential equations with jumps}, Stochastics and Stochastics Reports, {\bf 65} (1998), 111-125.

\bibitem[OS]{OS} B. Oksendal and A. Sulem, {\em Applied Stochastic Control of Jump Diffusions}, Springer-Verlag, Berlin, 2005.

\bibitem[PP]{PP} E. Pardoux and S. Peng, {\em Backward stochastic differential equations and quasiliner parabolic partial differential equations}, Lecture Notes in CIS, {\bf 176} (1992), 200-217.

\bibitem[R]{R} S. Ramasubramanian, {\em Reflected backward stochastic differetial equations in an orthant}, Proceedings of the Indian Academy of Science, {\bf 112} (2002), 347-360
\bibitem[S]{S} Y. Saisho, {\em Stochastic differential equations for multi-dimensional domain
with reflecting boundary}, Probability Theory and Related Fields, {\bf 74} (1987), 455-477.

\bibitem[T]{T} H. Tanaka, {\em Stochastic differential equations with reflecting boundary conditions in convex regions}, Hiroshima Mathematical Journal, {\bf 9} (1979), 163-177.
1987.

\bibitem[TL]{TL} S. Tang and X. Li, {\em Necessary conditions for optimal control of stochastic systems with random jumps}, SIAM Journal of Control and Optimization, {\bf 32} (1994), 1447-1475.
\end{thebibliography}
\end{document}